\documentclass[onecolumn]{article}
\usepackage{cite}

\usepackage{dsfont}
\usepackage{amsmath}
\usepackage{amssymb}
\usepackage{mathrsfs}
\usepackage{enumerate}
\usepackage{graphicx}
\usepackage{fullpage}

\usepackage{color}
\providecommand{\blue}[1]{\color{black}{#1}\color{black}\hspace{0pt}}

\newtheorem{theorem}{Theorem}
\newtheorem{corollary}[theorem]{Corollary}
\newtheorem{proposition}[theorem]{Proposition}

\newtheorem{definition}[theorem]{Definition}
\newtheorem{example}[theorem]{Example}

\newtheorem{remark}[theorem]{Remark}

\newtheorem{problem}[theorem]{Problem}

\newenvironment{proof}{{\it Proof :~}}{\hfill$\blacksquare$\\}

\def\Tmin{T_{\rm min}}
\def\Tmax{T_{\rm max}}

\DeclareMathOperator*{\col}{col}

\DeclareMathOperator*{\diag}{diag}

\DeclareMathOperator{\eps}{\varepsilon}

\DeclareMathOperator{\Deltabc}{\boldsymbol{\Delta_c}}
\DeclareMathOperator{\Deltabd}{\boldsymbol{\Delta_d}}

\begin{document}

\title{Stability and $L_1/\ell_1$-to-$L_1/\ell_1$ performance analysis of uncertain impulsive linear positive systems with applications to the interval observation of impulsive and switched systems with constant delays}

\author{Corentin Briat\thanks{corentin@briat.info; www.briat.info}}

\maketitle

\begin{abstract}
Solutions to the interval observation problem for delayed impulsive and switched systems with $L_1$-performance are provided. The approach is based on first obtaining stability and  $L_1/\ell_1$-to-$L_1/\ell_1$ performance analysis conditions for uncertain linear positive impulsive systems in linear fractional form with norm-bounded uncertainties using a scaled small-gain argument involving time-varying $D$-scalings. Both range and minimum dwell-time conditions are formulated -- the case of constant and maximum dwell-times can be directly obtained as corollaries. The conditions are stated as timer/clock-dependent conditions taking the form of infinite-dimensional linear programs that can be relaxed into finite-dimensional ones using polynomial optimization techniques. It is notably shown that under certain conditions, the scalings can be eliminated from the stability conditions to yield equivalent stability conditions on the so-called \emph{worst-case system}, which is obtained by replacing the uncertainties by the identity matrix. These conditions are then applied to the special case of linear positive systems with delays, where the delays are considered as uncertainties, similarly to as in \cite{Briat:16b}. As before, under certain conditions, the scalings can be eliminated from the conditions to obtain conditions on the worst-case system, coinciding here with the zero-delay system -- a result that is consistent with all the existing ones in the literature on linear positive systems with delays. Finally, the case of switched systems with delays is considered. The approach also encompasses standard continuous-time and discrete-time systems, possibly with delays and the results are flexible enough to be extended to cope with multiple delays, time-varying delays, distributed/neutral delays and any other types of uncertain systems that can be represented as a feedback interconnection of a known system with an uncertainty.

\end{abstract}

\section{Introduction}

The interval observation problem amounts to finding a pair of observers to estimate an upper bound and a lower bound on the state of a given system. Interval observers have been first proposed in the context of state estimation in biological processes in  \cite{Gouze:00}. Over the past recent years, this problem witnessed an increase in its popularity and various methodologies for their design in different setups have been proposed. To cite a few, those observers have been obtained for systems with inputs  \cite{Mazenc:11,Briat:15g}, linear systems \cite{Mazenc:12,Combastel:13,Cacace:15}, time-varying systems  \cite{Thabet:14}, delay systems  \cite{Efimov:13c,Briat:15g}, impulsive systems \cite{Degue:16nolcos}, uncertain/LPV systems  \cite{Efimov:13b,Chebotarev:15,Bolajraf:15}, discrete-time systems  \cite{Mazenc:13,Briat:15g}, systems with samplings  \cite{Mazenc:14b,Efimov:16}, impulsive systems  \cite{Degue:16nolcos,Briat:17ifacObs,Briat:18:IntImp}, switched systems \cite{Rabehi:17,Ethabet:17,Briat:18:IntImp} and Markovian jump systems  \cite{Briat:18:IntMarkov}.

An interesting feature of the underlying theory behind the design of interval observers lies in the fact that the overall aim is to design the observers in such a way that the errors dynamics are governed by positive systems. In this regard, the tools from the, now very rich, positive systems theory \cite{Farina:00} can be applied to the design of interval observers. In particular, the design of structured state-feedback controllers is convex in this setup  \cite{Aitrami:07,Briat:11h}, the $L_p$-gains for $p=1,2,\infty$ of linear positive systems can be exactly computed by solving linear or semidefinite programs  \cite{Ebihara:11,Briat:11g,Briat:11h}, optimal state-feedback controllers and observers gains enjoy an interesting invariance property with respect to the input and output matrices of a linear positive system, respectively  \cite{Ebihara:12,Briat:15g}. Linear positive systems with discrete-delays are also stable provided that their zero-delay counterparts are also stable; see e.g. \cite{Haddad:04,Zhu:15,Ebihara:17,Briat:11h,AitRami:09,Shen:14,Shen:15,Shen:15b}. In particular, all those exact results have been shown to be consequences of small-gain arguments using various $L_p$-gains in  \cite{Briat:16b} by exploiting the robust analysis results reported in \cite{Briat:11g,Ebihara:11,Colombino:15}. Some extensions have also been provided, notably pertaining on the analysis of linear positive systems with time-varying delays.

\blue{Before delving into the objective and the contributions of this paper, it seems important to motivate the consideration of uncertain linear impulsive systems. Impulsive systems are, in fact, a powerful class of hybrid systems that can be used to model a wide variety of real-world processes such as systems with impacts \cite{Goebel:12} (e.g. bouncing ball, robots with impact physical constraints such as walking robots, etc.), switched systems \cite{Goebel:12,Briat:13d,Briat:15i,Briat:16c} (which can be used to model, among others, networked control systems), sampled-data systems \cite{Sivashankar:94,Naghshtabrizi:08,Goebel:12,Briat:13d,Briat:15i}, LPV systems \cite{Wirth:05,Briat:17ifacLPV,Briat:17AutomaticaLPV}. Stochastic versions of impulsive systems \cite{Teel:14} can be used to model stochastic sampled-data systems \cite{Briat:15i}, Markov jump linear systems (which find applications in networked control systems and fault detection problems \cite{Sauter:13}) and many other types of systems. Their importance and versatility clearly motivates their consideration in this paper.}

\blue{In the previous papers from the same author \cite{Briat:17ifacObs,Briat:18:IntImp}, the interval observation problem for impulsive systems without and with performance constraints was considered. The objective of this paper is, therefore, to extend the results to the case of uncertain/delay impulsive/switched systems.  The approach developed in  \cite{Briat:16c} by the same author upon which the results in \cite{Briat:17ifacObs,Briat:18:IntImp} are based on is not directly applicable to time-delay systems. However, the approach is fully generalizable to uncertain systems written in linear fractional form (see e.g.  \cite{Zhou:96,Briat:book1}) whose stability analysis is amenable to problems that can be solved using input/output methods. It is worth mentioning here that input/output approaches provide a formidable tool for the analysis of complex systems; see e.g. \cite{Willems:72,Megretski:93,RantzerMegretski:97,Iwasaki:98a,Scherer:01,Kao:07,Cantoni:12,Scherer:12}. They have been shown to be often exact in the context of linear positive systems; see e.g. \cite{Briat:11h,Colombino:15,Colombino:15b,Briat:16b}.}

The first part of the paper is devoted to the development of stability and $L_1/\ell_1$ performance analysis conditions for linear uncertain positive impulsive systems in linear fractional form. In this context, we consider uncertainties both in the continuous-time and the discrete-time dynamics of the system. To reduce the conservatism of the approach, we use $D$-scalings for dealing with those uncertainties, as is customary in the field of robust analysis and control; see e.g.  \cite{Packard:93,Packard:94a,Apkarian:95a}. Note that in the context of positive systems, it is not restrictive to consider diagonal scalings \cite{Colombino:15} as opposed to general linear systems. A particularity is that the continuous-time part of the stability conditions depends on the timer variable $\tau$ (see e.g.  \cite{Goebel:12,Briat:13d,Briat:14f,Briat:15i}) \blue{which suggests that the scalings should be made timer-dependent in order to reduce the conservatism of the approach.} However, this leads to an additional difficulty that needs to be carefully taken care of. This difficulty arises from the fact that the scalings need to satisfy a commutation property with the uncertainties. When the uncertainties are memoryless (such as parametric uncertainties), the timer-dependent scalings trivially commute with the uncertainties. However, when they are not memoryless (dynamic uncertainties or delays), some extra work is needed to characterize the conditions under which the commutativity property holds. Four results on the stability of a linear uncertain positive impulsive system are obtained. The two first ones consider a range-dwell time condition in which either the continuous-time scaling is constrained (e.g. some entries on the diagonal are identical with each other or the scaling is constant) or unconstrained (and can therefore be eliminated). The two last ones consider a minimum dwell-time condition and again the cases where the continuous-time scaling is constrained or unconstrained. In the unconstrained scaling case, the conditions reduce to a stability and performance analysis condition on the worst-case system where the uncertainties are replaced by the identity matrix. All the obtained conditions are stated as infinite-dimensional linear programs that can be solved using polynomial programming techniques such as sum of squares programming \cite{Parrilo:00,sosopt,sostools3} , linear programming with the use of the Handelman's theorem \cite{Handelman:88,Briat:11h,Briat:15f}, or via so-called the piecewise linear approach \cite{Allerhand:11,Allerhand:13,Xiang:15a}.

The second part of the paper focuses on the stability of linear positive impulsive systems with constant delays. It is important to stress that stability analysis conditions for this class of systems have been obtained in the literature using Lyapunov-Krasovskii functionals or Lyapunov-Razumikhin functions; see e.g. \cite{Wang:14,Liu:14,Li:17,Hu:17}. \blue{However, this is the first time that delays are both considered in the continuous-time and the discrete-time dynamics of the system and that the stability of such systems is established using an input/output approach.} The first issue that needs to be addressed is to establish under what condition the timer-dependent scalings commute with the constant delay operator. We prove that this is the case if and only if the dwell-time sequence exhibits a periodic behavior with a period related to the continuous-time delay value. Unfortunately, this condition is very restrictive and, as a consequence, the use of timer-dependent scalings is, in general, not possible when dealing with delays as uncertainties. In this regard, constant $D$-scalings are most likely to be considered in practice. Two results are given in the context of a range dwell-time condition. The first one considers a constant (i.e. timer-independent) scaling whereas the second considers a timer-dependent one, which can be fully eliminated from the conditions. The same set of results is obtained in the minimum dwell-time case. Interestingly, we exhibit in both cases an interesting fact  that is recurrent in the analysis of linear positive systems. We show that, in the timer-dependent scaling case, the stability conditions are satisfied if and only if the same conditions are satisfied for the "zero-delay system" (the system obtained by setting the delays to 0). In other words, the stability of the zero-delay system implies that of the system with delay, and this is true for any value of the delay. In the context of constant scalings, the stability of the zero-delay system is only necessary.

The third and fourth parts of the paper are devoted to the application of those results to the design of interval observers for linear impulsive systems with delays and linear switched systems with delays, respectively. This is, to the author's best knowledge, the first time that such conditions are obtained. In the case of impulsive systems, the cases of constrained and unconstrained scalings are considered in both the range and minimum dwell-time setting. In the case of switched systems, only the minimum dwell-time case is treated. The obtained design conditions are stated as infinite-dimensional linear programs that can be easily solved using polynomial programming techniques such as sum of squares programming. Explicit values for the gains of the observer can be extracted from the solution to the optimization problems.\\

\noindent\textbf{Notations.} The set of integers greater (or equal) to $n\in\mathbb{Z}$ is denoted by $\mathbb{Z}_{>n}$ ($\mathbb{Z}_{\ge n}$). The cones of positive and nonnegative vectors of dimension $n$ are denoted by $\mathbb{R}_{>0}^n$ and $\mathbb{R}_{\ge0}^n$, respectively. The set of diagonal matrices of dimension $n$ is denoted by $\mathbb{D}^n$ and the subset of those being positive definite is denoted by $\mathbb{D}_{\succ0}^n$. The $n$-dimensional vector of ones is denoted by $\mathds{1}_n$. The dimension will be often omitted as it is obvious from the context. For some scalars $x_1,\ldots,x_n$ or some vector $x=(x_1,\ldots,x_n)$, $\diag_{i=1}^n(x_i)$ and $\diag(x)$ both denote the matrix with the diagonal entries $x_1,\ldots,x_n$ whereas $\col_{i=1}^n(x_i)$ and $\col(x)$  both create a vector by vertically stacking them with $x_1$ on the top.\\

\noindent\textbf{Outline.} The paper is structured as follows. Section \ref{sec:prel} states some preliminary definitions and results. In Section \ref{sec:lups}, stability and performance analysis conditions for uncertain linear impulsive positive systems are obtained. These conditions are then specialized to the subcase of linear impulsive positive systems with delays in Section \ref{sec:lpsd}. Finally, conditions for the design of interval observers for linear impulsive time-delay systems and linear switched time-delay systems are formulated in Section \ref{sec:IIimp} and Section \ref{sec:IIsw}, respectively. Numerical examples are provided in the related sections.

\section{Preliminaries}\label{sec:prel}

We consider in this section the following \blue{class of uncertain timer-dependent} impulsive system
\begin{equation}\label{eq:mainsyst2}
\begin{array}{rcl}
  \begin{bmatrix}
    \dot{x}(t_k+\tau)\\
     z_{c,\Delta}(t_k+\tau)\\
     z_c(t_k+\tau)
  \end{bmatrix}&=&\begin{bmatrix}
    A(\tau) & G_c(\tau) & E_c(\tau)\\
    C_{c,\Delta} & H_{c,\Delta} & F_{c,\Delta}\\
    C_c & H_c & F_c
  \end{bmatrix}\begin{bmatrix}
    x(t_k+\tau)\\
    w_{c,\Delta}(t_k+\tau)\\
    w_c(t_k+\tau)
  \end{bmatrix},\tau\in(0,T_k],k\in\mathbb{Z}_{\ge0}\\
  \begin{bmatrix}
  x(t_k^+)\\
  z_{d,\Delta}(k)\\
  z_d(k)
  \end{bmatrix}&=&\begin{bmatrix}
    J & G_d & E_d\\
    C_{d,\Delta} & H_{d,\Delta} & F_{d,\Delta}\\
    C_d & H_d & F_d
  \end{bmatrix}\begin{bmatrix}
    x(t_k)\\
    w_{d,\Delta}(k)\\
    w_d(k)
  \end{bmatrix},\ k\in\mathbb{Z}_{\ge 1}\\
  w_{c,\Delta}(t_k+\tau)&=&(\Delta_cz_{c,\Delta})(t_k+\tau)\\
  w_{d,\Delta}(k)&=&(\Delta_dz_{d,\Delta})(k)\\
  x(0)&=&x(0^+)=x_0\\
\end{array}
\end{equation}
where $x,x_0\in\mathbb{R}^n$, $w_c\in\mathbb{R}^{p_c}$, $w_d\in\mathbb{R}^{p_d}$, $y_c\in\mathbb{R}^{q_c}$ and $y_d\in\mathbb{R}^{q_d}$ are the state of the system, the initial condition, the continuous-time exogenous input, the discrete-time exogenous input, the continuous-time performance output and the discrete-time performance output, respectively. The pair of signals $z_{c,\Delta},w_{c,\Delta}\in\mathbb{R}^{n_{c,\Delta}}$ and $z_{d,\Delta},w_{d,\Delta}\in\mathbb{R}^{n_{d,\Delta}}$ are the uncertainty channels and the operators $\Delta_c$ and $\Delta_d$ are bounded operators (more details will be given later). Above, $x(t_k^+):=\lim_{s\downarrow t_k}x(s)$ and the matrix-valued functions $A(\tau)\in\mathbb{R}^{n\times n}$ and $E(\tau)\in\mathbb{R}^p$ are continuous. The sequence of impulse times $\{t_k\}_{k\ge1}$ is assumed to verify the properties: (a) $T_k:=t_{k+1}-t_k>0$ for all $k\in\mathbb{Z}_{\ge0}$ (no jump at $t_0=0$) and (b) $t_k\to\infty$ as $k\to\infty$. When all the above properties hold, the solution of the system \eqref{eq:mainsyst2} exists for all times.

We have the following result regarding the state positivity of the impulsive system \eqref{eq:mainsyst2}.
\begin{proposition}
The following statements are equivalent:
\begin{enumerate}[(a)]
  \item The system \eqref{eq:mainsyst2} is internally positive, i.e. for any $x_0\ge0$, $w_c(t),w_{c,\Delta}(t)\ge0$ and $w_d(k),w_{d,\Delta}(k)\ge0$, we have that $x(t),z_c(t),z_{c,\Delta}(t)\ge0$ for all $t\ge0$ and $z_d(k),z_{d,\Delta}(k)\ge0$ for all $k\in\mathbb{Z}_{\ge0}$.
  \item The matrix-valued function $A(\tau)$ is Metzler for all $\tau\ge0$, the matrix-valued functions $E_c(\tau)$ and $G_c(\tau)$ are nonnegative for all $\tau\ge0$ and the matrices $J,G_d,E_d,C_{c,\Delta},H_{c,\Delta},F_{c,\Delta},C_{d,\Delta},H_{d,\Delta},F_{d,\Delta},C_c,H_c,F_c,C_d,H_d,F_d$ are nonnegative.
\end{enumerate}
\end{proposition}
\blue{\begin{proof}
  The proof is based on the combination of the positivity conditions for both continuous-time and discrete-time systems; see e.g. \cite{Farina:00}.
\end{proof}}

We recall now the concept of hybrid $L_1/\ell_1$-gain introduced in \cite{Briat:18:IntImp}:
\begin{definition}
We say that the system \eqref{eq:mainsyst2} has a hybrid $L_1/\ell_1$-gain of at most $\gamma$ if for all $w_c\in L_1$ and $w_d\in\ell_1$, we have that
  \begin{equation}
    ||z_c||_{L_1}+||z_d||_{\ell_1}<\gamma(||w_c||_{L_1}+||w_d||_{\ell_1})+v(||x_0||)
  \end{equation}
  for some increasing function $v$ verifying $v(0)=0$ and $v(s)\to\infty$ as $s\to\infty$.
\end{definition}

We now define the sets for our uncertainties:
\begin{definition}
  The uncertain operators $\Delta_c$ and $\Delta_d$ are assumed to belong to the sets
  \begin{equation}
    \Delta_c\in\Deltabc:=\left\{\Delta:L_1\mapsto L_1:||\Delta||_{L_1-L_1}\le1\right\}
  \end{equation}
  where $||\Delta||_{L_1-L_1}$ denotes the $L_1$-gain of the operator $\Delta$, and
  \begin{equation}
    \Delta_d\in\Deltabd:=\left\{\Delta:\ell_1\mapsto \ell_1:||\Delta||_{\ell_1-\ell_1}\le1\right\}
  \end{equation}
  where $||\Delta||_{\ell_1-\ell_1}$ denotes the $\ell_1$-gain of the operator $\Delta$.
\end{definition}
Note that we do not restrict the operators to map positive inputs to positive outputs since what matters here is that the operators see a positive system, that is, that the maps $w_{c,\Delta}\mapsto z_{c,\Delta}$ and $w_{d,\Delta}\mapsto z_{d,\Delta}$ be positive. In such a case, and as pointed out in \cite{Colombino:15}, the worst case operator in the above set is necessarily going to be a positive one.

As customary in the input/output setting (see e.g. \cite{Packard:93}), we recall now the concept of $D$-scalings:
\begin{definition}
We define the set of timer-dependent continuous-time $D$-scalings as
\begin{equation}
    \mathcal{S}_c:=\left\{S:[0,T]\mapsto\mathbb{D}_{\succ0}^{n_{\Delta,c}}|S\circ \Delta_c=\Delta_c\circ S\right\}
  \end{equation}
  for some time $T>0$ and where $\circ$ is the composition operator. The set of discrete-time $D$-scalings is defined as
  \begin{equation}
    \mathcal{S}_c:=\left\{S\in\mathbb{D}_{\succ0}^{n_{\Delta,d}}|S\circ \Delta_d=\Delta_d\circ S\right\}.
  \end{equation}
\end{definition}

\begin{example}
  If, for instance, the operator $\Delta_c=M$ is the multiplication operator taking the form $M=\diag(\theta_1I_{n_1},\ldots,\theta_KI_{n_K})$ where $K$ is the number of distinct parameters and $n_i$ is the occurrence of the parameter $i$ in the diagonal matrix $M$. In this case, the set of scalings is simply defined as all the mappings $S:[0,T]\mapsto\mathbb{D}_{\succ0}^{n_1+\ldots+n_K}$.
\end{example}


\section{Stability and performance analysis of linear uncertain positive systems}\label{sec:lups}

\blue{The objective of this section is to provide stability and performance criteria for systems of the form \eqref{eq:mainsyst2}. Those criteria are novel and extend previously obtained ones on the stability analysis \cite{Briat:16c} and the $L_1/\ell_1$-performance \cite{Briat:18:IntImp} to the case of uncertain systems in LFT-form.} First, conditions for the stability and the hybrid $L_1/\ell_1$ performance analysis for the system \eqref{eq:mainsyst2} are obtained under a range dwell-time constraint. Then, an analogous result is obtained with the difference that a minimum dwell-time constraint is considered. The constant and maximum dwell-time cases can be easily obtained as corollaries or simple adaptations of those results. Finally, results in the case of unconstrained scalings are provided. In such a case, the scalings can be fully eliminated to obtain a reduced set of conditions that are, as it turns out, identical to conditions that would have been obtained by replacing the uncertainties by identity matrices, illustrating then the fact that the worst-case operators coincide with the identity matrix.

\subsection{Range dwell-time stability and performance condition}

We first address the range dwell-time case, that is, the case where the dwell-time values $T_k$, $k\in\mathbb{Z}_{\ge0}$, belong to the interval $[\Tmin,\Tmax]$ where $0\le \Tmin\blue{\le }\Tmax<\infty$. Stability and performance conditions are stated in the following result:
\begin{theorem}[Range dwell-time]\label{th:rangeDT:general}
Assume that the system \eqref{eq:mainsyst2} is internally positive and that there exist a differentiable vector-valued function $\zeta:[0,\Tmax ]\mapsto\mathbb{R}^n$, $\zeta(0)>0$, a vector-valued function $\mu_c:[0,\Tmax ]\mapsto\mathbb{R}^{n_{c,\Delta}}$, a vector $\mu_d\in\mathbb{R}^{n_{d,\Delta}}$ and scalars $\epsilon,\gamma>0$ such that the conditions $\diag(\mu_c)\in\mathcal{S}_c$, $\diag(\mu_d)\in\mathcal{S}_d$,
\begin{equation}\label{eq:RangeDT1}
    \begin{bmatrix}
      \dot{\zeta}(\tau)\\
      0\\
      0
    \end{bmatrix}^T+\begin{bmatrix}
      \zeta(\tau)\\
      \mu_c(\tau)\\
      \mathds{1}
    \end{bmatrix}^T\begin{bmatrix}
      A(\tau) & G_c(\tau) & E_c(\tau)\\
      C_{c,\Delta} & H_{c,\Delta}-I & F_{c,\Delta}\\
      C_c & H_c & F_c
    \end{bmatrix}\le\begin{bmatrix}
    0\\0\\ \gamma\mathds{1}
    \end{bmatrix}^T
  \end{equation}
  and
    \begin{equation}\label{eq:RangeDT2}
    \begin{bmatrix}
      -\zeta(\theta)\\
      0\\
      0
    \end{bmatrix}^T+\begin{bmatrix}
      \zeta(0)\\
      \mu_d\\
      \mathds{1}
    \end{bmatrix}^T\begin{bmatrix}
      J & G_d & E_d\\
      C_{d,\Delta} & H_{d,\Delta}-I & F_{d,\Delta}\\
      C_d & H_d & F_d
    \end{bmatrix}\le\begin{bmatrix}
      -\epsilon\mathds{1}^T\\0\\ \gamma\mathds{1}
    \end{bmatrix}^T
  \end{equation}
hold for all $\tau\in[0,\Tmax ]$ and all $\theta\in[\Tmin ,\Tmax ]$. Then, the system \eqref{eq:mainsyst2} is asymptotically stable for all $\Delta_c\in\Deltabc$ and $\Delta_d\in\Deltabd$ under the range dwell-time condition $[\Tmin ,\Tmax ]$. Moreover, the mapping $(w_c,w_d)\mapsto(z_c,z_d)$ has a hybrid $L_1/\ell_1$-gain of at most $\gamma$.
\end{theorem}
\begin{proof}
Let $S_c:=\diag(\mu_c)$ (i.e. $\mathds{1}^TS_c=\mu_c$) and $S_d:=\diag(\mu_d)$. 
Then, multiply from the right the conditions \eqref{eq:RangeDT1} and \eqref{eq:RangeDT2} by $\col(x(t_k+\tau),w_{c,\Delta}(t_k+\tau),w_c(t_k+\tau))$ and $\col(x(t_k),w_{d,\Delta}(k),w_d(k))$, respectively. Grouping the terms together yields
\begin{equation}
\begin{array}{rcl}
  \dot\zeta(\tau)^Tx(t_k+\tau)+\zeta(\tau)^T\dot{x}(t_k+\tau)-\mathds{1}^TS_c(\tau)w_{c,\Delta}(t_k+\tau)+\mathds{1}^TS_c(\tau)z_{c,\Delta}(t_k+\tau)\\
  -\gamma\mathds{1}^Tw_c(t_k+\tau)+\mathds{1}^Tz_c(t_k+\tau)\le0
\end{array}
\end{equation}
Letting $V_k(\tau,x)=\lambda(\tau)^Tx(t_k+\tau)$ and integrating the above inequality from 0 to $T_k$ yields
\begin{equation}
\begin{array}{rcl}
  V_k(T_k,x)-V_k(0,x)+\int_{0}^{T_k}\left[R_k^{c,\Delta}(\tau)+R_k^{c}(\tau)\right]d\tau\le0
\end{array}
\end{equation}
where $R_k^{c,\Delta}(\tau):=\mathds{1}^TS_c(\tau)z_{c,\Delta}(t_k+\tau)-\mathds{1}^TS_c(\tau)w_{c,\Delta}(t_k+\tau)$ and $R_k^{c}(\tau):=\mathds{1}^Tz_c(t_k+\tau)-\gamma\mathds{1}^Tw_c(t_k+\tau)$. Similar calculations for the condition \eqref{eq:RangeDT2} yield
\begin{equation}
  V_{k+1}(0,x)-V_{k}(T_k,x)+R_k^{d,\Delta}+R_k^{d}+\epsilon\mathds{1}^Tx(t_k)\le0
\end{equation}
where $R_k^{d,\Delta}(\tau):=\mathds{1}^TS_dz_{d,\Delta}(k)-\mathds{1}^TS_dw_{d,\Delta}(k)$ and $R_k^{d}:=\mathds{1}^Tz_d(k)-\gamma\mathds{1}^Tw_d(k)$. Combining those expressions yields
\begin{equation}
\begin{array}{rcl}
  V_{k+1}(0,x)-V_k(0,x)+\int_{0}^{T_k}\left[R_k^{c,\Delta}(\tau)+R_k^{c}(\tau)\right]d\tau+R_k^{d,\Delta}+R_k^{d}+\epsilon\mathds{1}^Tx(t_k)\le0.
\end{array}
\end{equation}
Summing over $k$ from 0 to $\infty$ and using the fact since the system is stable and the inputs are in $L_1/\ell_1$, then $x$ goes to 0 as $t\to\infty$. Hence, we have that
\begin{equation}
\begin{array}{rcl}
  -V_0(0,x)+\sum_{k=0}^\infty\left(\int_{0}^{T_k}\left[R_k^{c,\Delta}(\tau)+R_k^{c}(\tau)\right]d\tau+R_k^{d,\Delta}+R_k^{d}\right)+\tilde\epsilon\mathds{1}^Tx(t_k)\le0.
\end{array}
\end{equation}
where $\tilde\epsilon:=\sum_{k=0}^\infty\epsilon\mathds{1}^Tx(t_k)$. Using the fact that, by definition $\sum_{k=0}^\infty\int_{0}^{T_k}R_k^{c,\Delta}(\tau)d\tau\ge0$ and $\sum_{k=0}^\infty R_k^{d,\Delta}\ge0$, then this implies that
\begin{equation}
\begin{array}{rcl}
  -V_0(0,x)+\sum_{k=0}^\infty\left(\int_{0}^{T_k}R_k^{c}(\tau)d\tau+R_k^{d}\right)+\tilde\epsilon\mathds{1}^Tx(t_k)\le0.
\end{array}
\end{equation}
Assuming now zero initial conditions and considering the fact that $\epsilon>0$ can be chosen arbitrarily small, then we get that
\begin{equation}
\begin{array}{rcl}
  \int_{0}^{\infty}\left[\mathds{1}^Tz_c(s)-\gamma\mathds{1}^Tw_c(s)\right]d\tau+\sum_{k=0}^\infty\left[\mathds{1}^Tz_d(k)-\gamma\mathds{1}^Tw_d(k) \right]+\tilde\epsilon\mathds{1}^Tx(t_k)\le0.
\end{array}
\end{equation}
and, hence, that
\begin{equation}
\begin{array}{rcl}
  ||z_c||_{L_1}+||z_d||_{\ell_1}<\gamma\left(||w_c||_{L_1}+||w_d||_{\ell_1}\right).
\end{array}
\end{equation}
This proves the result.
\end{proof}

\subsection{Minimum dwell-time stability and performance condition}

We address now the minimum dwell-time case, that is, the case where the dwell-time values $T_k$, $k\in\mathbb{Z}_{\ge0}$, belong to the interval $[\bar T,\infty)$ where $0<\bar T$. Stability and performance conditions are stated in the following result:
\begin{theorem}[Minimum dwell-time]\label{th:minimumDT:general}
Let us assume here that the system \eqref{eq:mainsyst2} is internally positive and that the matrices of the system are such that they remain constant for all values $\tau\ge\bar T$. Assume further that there exist a differentiable vector-valued function $\zeta:[0,\bar T]\mapsto\mathbb{R}^n$, $\zeta(0)>0,\zeta(\bar T)>0$, a vector-valued function $\mu_c:[0,\bar T]\mapsto\mathbb{R}^{n_{c,\Delta}}$, a vector $\mu_d\in\mathbb{R}^{n_{d,\Delta}}$ and scalars $\epsilon,\gamma>0$ such that the conditions $\diag(\mu_c)\in\mathcal{S}_c$, $\diag(\mu_d)\in\mathcal{S}_d$,
\begin{equation}\label{eq:minDT1}
    \begin{bmatrix}
      \zeta(\bar T)\\
      \mu_c(\bar T)\\
      \mathds{1}
    \end{bmatrix}^T\begin{bmatrix}
      A(\bar T) & G_c(\bar T) & E_c(\bar T)\\
      C_{c,\Delta} & H_{c,\Delta}-I & F_{c,\Delta}\\
      C_c & H_c & F_c
    \end{bmatrix}\le\begin{bmatrix}
    -\epsilon\mathds{1}\\0\\ \gamma\mathds{1}
    \end{bmatrix}^T
  \end{equation}
\begin{equation}\label{eq:minDT2}
    \begin{bmatrix}
      \dot{\zeta}(\tau)\\
      0\\
      0
    \end{bmatrix}^T+\begin{bmatrix}
      \zeta(\tau)\\
      \mu_c(\tau)\\
      \mathds{1}
    \end{bmatrix}^T\begin{bmatrix}
      A(\tau) & G_c(\tau) & E_c(\tau)\\
      C_{c,\Delta} & H_{c,\Delta}-I & F_{c,\Delta}\\
      C_c & H_c & F_c
    \end{bmatrix}\le\begin{bmatrix}
    0\\0\\ \gamma\mathds{1}
    \end{bmatrix}^T
  \end{equation}
  and
    \begin{equation}\label{eq:minDT3}
    \begin{bmatrix}
      -\zeta(\theta)\\
      0\\
      0
    \end{bmatrix}^T+\begin{bmatrix}
      \zeta(0)\\
      \mu_d\\
      \mathds{1}
    \end{bmatrix}^T\begin{bmatrix}
      J & G_d & E_d\\
      C_{d,\Delta} & H_{d,\Delta}-I & F_{d,\Delta}\\
      C_d & H_d & F_d
    \end{bmatrix}\le\begin{bmatrix}
      -\epsilon\mathds{1}^T\\0\\ \gamma\mathds{1}
    \end{bmatrix}^T
  \end{equation}
hold for all $\tau\in[0,\bar T ]$. Then, the system \eqref{eq:mainsyst2} is asymptotically stable for all $\Delta_c\in\Deltabc$ and $\Delta_d\in\Deltabd$ under the minimum dwell-time condition $\bar T$. Moreover, the mapping $(w_c,w_d)\mapsto(z_c,z_d)$ has a hybrid $L_1/\ell_1$-gain of at most $\gamma$.
\end{theorem}
\begin{proof}
  The proof is based on Theorem \ref{th:rangeDT:general} where we have considered $\zeta$ and $\mu_c$ such that they remain constant for all values of $\tau\ge\bar T$.
\end{proof}

\subsection{The unconstrained scalings case}

It seems interesting here to discuss the case where the scalings are unconstrained. By "unconstrained", it is meant here that the set of continuous-time $D$-scalings coincides with the set of all maps from $[0,T]$ to $\mathbb{D}_{\succ0}^{n_{c,\Delta}}$ and the set of discrete-time $D$-scalings is simply the set $\mathbb{D}_{\succ0}^{n_{c,\Delta}}$. A necessary condition for this fact to hold is that $\Delta_c$ and $\Delta_d$ be diagonal. This is notably the case when parametric uncertainties or delay operators are considered. In this very interesting case, the scalings can be eliminated from the conditions to get equivalent ones characterizing stability of the uncertain system. To this aim, let us assume that the operators are any diagonal bounded operator with unit $L_1$- and $\ell_1$-gains admitting unconstrained scalings. It is interesting to note that the operators need not be restricted to be positive only as pointed out in \cite{Colombino:15} since all what matters is that the operators see a positive system; i.e. the maps $w_{c,\Delta}\mapsto z_{c,\Delta}$ and $w_{d,\Delta}\mapsto z_{d,\Delta}$ be positive. Moreover, in the same paper it is shown that the worst case operator in the unit-ball is the positive operator with unit gain. For example, the worst-case value for a scalar parameter in the closed unit-ball is one.

On the strength of the discussion above, we can state the following stability and performance result under a range dwell-time constraint:
\begin{theorem}[Range dwell-time]\label{th:rangeDT:free}
Assume that there exist a differentiable vector-valued function $\zeta:[0,\Tmax ]\mapsto\mathbb{R}^n$, $\zeta(0)>0$ and scalars $\epsilon,\gamma>0$ such that
\begin{equation}
    \begin{bmatrix}
      \dot{\zeta}(\tau)\\
      0
    \end{bmatrix}^T+\begin{bmatrix}
      \zeta(\tau)\\
      \mathds{1}
    \end{bmatrix}^T\left(\begin{bmatrix}
      A(\tau) & E_c(\tau)\\
      C_c & F_d
    \end{bmatrix}+\begin{bmatrix}
      G_c(\tau)\\
      H_c
    \end{bmatrix}(I-H_{c,\Delta})^{-1}\begin{bmatrix}
      C_{c,\Delta} &  F_{c,\Delta}
    \end{bmatrix}\right)\le\begin{bmatrix}
    0\\ \gamma\mathds{1}
    \end{bmatrix}^T
  \end{equation}
  and
    \begin{equation}
    \begin{bmatrix}
      -\zeta(\theta)\\
      0
    \end{bmatrix}^T+\begin{bmatrix}
      \zeta(0)\\
      \mathds{1}
    \end{bmatrix}^T\left(\begin{bmatrix}
      J & E_d\\
      C_d& F_d
    \end{bmatrix}+\begin{bmatrix}
      G_d\\
      H_d
    \end{bmatrix}(I-H_{d,\Delta})^{-1}\begin{bmatrix}
      C_{d,\Delta} &  F_{d,\Delta}
    \end{bmatrix}\right)\le\begin{bmatrix}
      -\epsilon\mathds{1}^T\\ \gamma\mathds{1}
    \end{bmatrix}^T
  \end{equation}
hold for all $\tau\in[0,\Tmax ]$ and all $\theta\in[\Tmin ,\Tmax ]$. Then, the system \eqref{eq:mainsyst2} is asymptotically stable for all $\Delta_c\in\Deltabc$ and $\Delta_d\in\Deltabd$ under the range dwell-time condition $[\Tmin ,\Tmax ]$. Moreover, the mapping $(w_c,w_d)\mapsto(z_c,z_d)$ has a hybrid $L_1/\ell_1$-gain of at most $\gamma$.
\end{theorem}
\blue{\begin{proof}
  In the unconstrained scaling case, we can solve for the scalings $\mu_c(\tau)$ and $\mu_d$ in the conditions \eqref{eq:RangeDT1} and \eqref{eq:RangeDT2}. To achieve this, it is enough to observe that since those scalings are positive, we need to take their smallest possible values that satisfy the conditions \eqref{eq:RangeDT1} and \eqref{eq:RangeDT2}. Hence, we can select these scalings such that
\begin{equation}
  \zeta(\tau)^TG_c(\tau)+\mu_c(\tau)^T(H_{c,\Delta}-I)+\mathds{1}^TH_c=0
\end{equation}
and
\begin{equation}
  \zeta(0)^TG_d+\mu_d^T(H_{d,\Delta}-I)+\mathds{1}^TH_d=0.
\end{equation}
Since the system is well-posed, then the Metzler matrices $H_{c,\Delta}-I$ and $H_{d,\Delta}-I$ are invertible. Moreover, for the above equality to hold, it is necessary that $\mu_c(\tau)^T(H_{c,\Delta}-I)$ and $\mu_d^T(H_{d,\Delta}-I)$ be componentwise negative, which is equivalent to saying that the matrices $H_{c,\Delta}-I$ and $H_{d,\Delta}-I$ are Hurwitz stable. Hence, their inverse is nonnegative \cite{Berman:94}. Solving for the scalings values in the above expressions then yields
\begin{equation}
  \mu_c(\tau)^T=-(\zeta(\tau)^TG_c(\tau)+\mathds{1}^TH_c)(H_{c,\Delta}-I)^{-1}>0
\end{equation}
and
\begin{equation}
  \mu_d^T=-(\zeta(0)^TG_d+\mathds{1}^TH_d)(H_{d,\Delta}-I)^{-1}>0.
\end{equation}
Substituting these values in \eqref{eq:RangeDT1} and \eqref{eq:RangeDT2} yields the conditions of Theorem \ref{th:rangeDT:free}. This proves the result.
\end{proof}}
Interestingly, we can see that the stability conditions of Theorem \ref{th:rangeDT:general} reduce to a stability condition where the uncertain operators are replaced by the identity matrix. This is consistent with the results in \cite{Briat:11g,Briat:11h} where this fact was pointed out for the first time and later analyzed in \cite{Colombino:15,Colombino:15b,Briat:15cdc}.

The following theorem states an analogous result in the minimum dwell-time case:
\begin{theorem}[Minimum dwell-time]\label{th:minimumDT:free}
Assume that there exist a differentiable vector-valued function $\zeta:[0,\Tmax ]\mapsto\mathbb{R}^n$, $\zeta(0)>0$ and scalars $\epsilon,\gamma>0$ such that
\begin{equation}
  \begin{bmatrix}
      \zeta(\bar T)\\
      \mathds{1}
    \end{bmatrix}^T\left(\begin{bmatrix}
      A(\bar T) & E_c(\bar T)\\
      C_c & F_d
    \end{bmatrix}+\begin{bmatrix}
      G_c(\tau)\\
      H_c
    \end{bmatrix}(I-H_{c,\Delta})^{-1}\begin{bmatrix}
      C_{c,\Delta} &  F_{c,\Delta}
    \end{bmatrix}\right)\le\begin{bmatrix}
    -\epsilon\mathds{1}\\ \gamma\mathds{1}
    \end{bmatrix}^T,
  \end{equation}
\begin{equation}
    \begin{bmatrix}
      \dot{\zeta}(\tau)\\
      0
    \end{bmatrix}^T+\begin{bmatrix}
      \zeta(\tau)\\
      \mathds{1}
    \end{bmatrix}^T\left(\begin{bmatrix}
      A(\tau) & E_c(\tau)\\
      C_c & F_d
    \end{bmatrix}+\begin{bmatrix}
      G_c(\tau)\\
      H_c
    \end{bmatrix}(I-H_{c,\Delta})^{-1}\begin{bmatrix}
      C_{c,\Delta} &  F_{c,\Delta}
    \end{bmatrix}\right)\le\begin{bmatrix}
    0\\ \gamma\mathds{1}
    \end{bmatrix}^T
  \end{equation}
  and
    \begin{equation}
    \begin{bmatrix}
      -\zeta(\bar T)\\
      0
    \end{bmatrix}^T+\begin{bmatrix}
      \zeta(0)\\
      \mathds{1}
    \end{bmatrix}^T\left(\begin{bmatrix}
      J & E_d\\
      C_d& F_d
    \end{bmatrix}+\begin{bmatrix}
      G_d\\
      H_d
    \end{bmatrix}(I-H_{d,\Delta})^{-1}\begin{bmatrix}
      C_{d,\Delta} &  F_{d,\Delta}
    \end{bmatrix}\right)\le\begin{bmatrix}
      -\epsilon\mathds{1}^T\\ \gamma\mathds{1}
    \end{bmatrix}^T
  \end{equation}
hold for all $\tau\in[0,\bar T]$. Then, the system \eqref{eq:mainsyst2} is asymptotically stable for all $\Delta_c\in\Deltabc$ and $\Delta_d\in\Deltabd$ under the minimum dwell-time condition $\bar T$. Moreover, the mapping $(w_c,w_d)\mapsto(z_c,z_d)$ has a hybrid $L_1/\ell_1$-gain of at most $\gamma$.
\end{theorem}

\blue{\subsection{Computational considerations}

Several methods can be used to check the conditions stated in Theorem \ref{th:rangeDT:general}, Theorem \ref{th:minimumDT:general}, Theorem \ref{th:rangeDT:free} and Theorem \ref{th:minimumDT:free}. The piecewise linear discretization approach \cite{Allerhand:11,Xiang:15a,Briat:16c} assumes that the decision variables are piecewise linear functions of their arguments and leads to a finite-dimensional linear program that can be checked using standard linear programming algorithms. Another possible approach is based on Handelman's Theorem \cite{Handelman:88} and also leads to a finite-dimensional program \cite{Briat:11h,Briat:16c}. We opt here for an approach based on Putinar's Positivstellensatz \cite{Putinar:93} and semidefinite programming \cite{Parrilo:00}\footnote{See \cite{Briat:16c} for a comparison of all these methods.}. Before stating the main result of the section, we need to define first some terminology. A multivariate polynomial $p(x)$ is said to be a sum-of-squares (SOS) polynomial if it can be written as $\textstyle p(x)=\sum_{i}q_i(x)^2$ for some polynomials $q_i(x)$. A polynomial matrix $p(x)\in\mathbb{R}^{n\times m}$ is said to \emph{componentwise sum-of-squares} (CSOS) if each of its entries is an SOS polynomial. Checking whether a polynomial is SOS can be exactly cast as a semidefinite program \cite{Parrilo:00,Chesi:10b} that can be easily solved using semidefinite programming solvers such as SeDuMi \cite{Sturm:01a}. The package SOSTOOLS \cite{sostools3} can be used to formulate and solve SOS programs in a convenient way.

Below is the SOS implementation of the conditions of Theorem \ref{th:minimumDT:general}:
\begin{proposition}\label{prop:SOS1}
  Let $d\in\mathbb{N}$, $\eps>0$ and $\epsilon>0$ be given and assume that there exist polynomials $\chi_i:\mathbb{R}\mapsto\mathbb{R}$, $i=1,\ldots,n$, $U_c:\mathbb{R}\mapsto\mathbb{R}^{n\times q_c}$, $\Psi_1:\mathbb{R}\mapsto\mathbb{R}^{n\times n}$,  $\Psi_2:\mathbb{R}\mapsto\mathbb{R}^{n\times q_c}$ and $\psi_1,\psi_2:\mathbb{R}\mapsto\mathbb{R}^{n}$, $\psi_3:\mathbb{R}\mapsto\mathbb{R}^{p_c}$ of degree $2d$, a matrix $U_d\in\mathbb{R}^{n\times q_d}$ and scalars $\alpha,\gamma\ge0$ such that
  \begin{enumerate}[(a)]
    \item $\Psi_1(\tau),\Psi_2(\tau),\psi_1(\tau),\psi_2(\tau)$ and $\psi_3(\tau)$ are CSOS,
    \item $X(0)-\epsilon I_n\ge0$ (or is CSOS),
    \item $X(\tau)A-U_c(\tau)C_c+\alpha I_n-\Psi_1(\tau)f(\tau)$ is CSOS,
    \item $X(0) J-U_d C_d\ge0$ (or is CSOS),
    \item $X(\tau)E_c-U_c(\tau)F_c-\Psi_2(\tau)f(\tau)$ is CSOS,
    \item $X(0)E_d-U_d F_d\ge0$ (or is CSOS),
    \item $-\mathds{1}^T\left[\dot{X}(\tau)+X(\tau)A-U_c(\tau)C_c\right]-f(\tau)\psi_1(\tau)^T$

    is CSOS,
    \item $-\mathds{1}^T\left[X(0) J-U_d C_d-X(\theta)+\eps I\right]-g(\theta)\psi_2(\theta)^T$

     is CSOS,
     %
     %
     \item $-\mathds{1}^T\left[X(\tau)E_c-U_c(\tau)F_c\right]+\gamma \mathds{1}^T-f(\tau)\psi_3(\tau)$ is CSOS,
     \item $-\mathds{1}^T\left[X(0)E_d-U_d F_d\right]+\gamma \mathds{1}^T\ge0$  (or is CSOS)
  \end{enumerate}
  where $X(\tau):=\diag_{i=1}^n(\chi_i(\tau))$, $f(\tau):=\tau(T_{max}-\tau)$ and $g(\theta):=(\theta-T_{min})(T_{max}-\theta)$.

  Then, the conditions of statement (b) of Theorem \ref{th:1} hold with the same $X(\tau)$, $U_c(\tau)$, $U_d$, $\alpha$, $\eps$ and $\gamma$.
\end{proposition}
\begin{proof}
The proof follows from the same arguments as the proof of Proposition 3.15 in \cite{Briat:16c}.
\end{proof}

\begin{remark}[Asymptotic exactness]
The above relaxation is asymptotically exact under very mild conditions \cite{Putinar:93} in the sense that if the original conditions of Theorem \ref{th:1} hold then we can find a degree $d$ for the polynomial variables for which the conditions in Proposition \ref{prop:SOS1} are feasible. This follows from the fact that a univariate polynomial is SOS if and only if it is positive semidefinite and that polynomials are dense in the set of continuous functions defined on a compact set. See \cite{Briat:16c} for more details.
\end{remark}}

\subsection{Example}

Let us consider here the system
\begin{equation}\label{eq:uncimp}
  \begin{array}{rcl}
    \dot{x}(t)&=&\left(\begin{bmatrix}
      -1 & 0\\
      1 & -3
    \end{bmatrix}+\dfrac{\theta(t)}{2-\theta(t)}\begin{bmatrix}
      0 & 1\\
      0 & 0
    \end{bmatrix}\right)x(t)+\begin{bmatrix}
      1\\
      0
    \end{bmatrix}w_c(t)\\
    z_c(t)&=&\begin{bmatrix}
      0 & 1
    \end{bmatrix}x(t)\\
    x(t_k^+)&=&2x(t_k)\\
    z_d(k)&=&\begin{bmatrix}
      0 & 1
    \end{bmatrix}x(t_k).
  \end{array}
\end{equation}
We can rewrite this system in the form \eqref{eq:mainsyst2} together with the matrices:
\begin{equation}
\begin{array}{c}
    A=\begin{bmatrix}
      -1 & 0\\
      1 & -3
    \end{bmatrix},G_c=\begin{bmatrix}
      0 & 2\alpha\\
      0 & 0
    \end{bmatrix},E_c=\begin{bmatrix}
      1\\
      0
    \end{bmatrix},\\
    C_{c,\Delta}=H_{c,\Delta}=\begin{bmatrix}
      1/2 & 0\\
      0 & 1/2
    \end{bmatrix},F_{c,\Delta}=0,\\
    C_c=\begin{bmatrix}
      0 & 1
    \end{bmatrix},H_c=F_c=0\\
    J = \begin{bmatrix}
      2 & 0\\
      0 & 2
    \end{bmatrix}, C_d=C_c
\end{array}
\end{equation}
and all the other matrices in the discrete-time part of the system are equal to 0. The parameter $\theta$ is assumed to take values in $[-1,1]$. The matrix  $A+G_c(I-H_{c,\Delta})^{-1}C_{c,\Delta}$ is Hurwitz stable, so this system is a candidate for a system that is stable under a minimum dwell-time constraint. Solving then the conditions of Theorem \ref{th:minimumDT:free}, we get the results depicted in Fig.~\ref{fig:minDT_L2_impunc} where the hybrid $L_1/\ell_1$-gain of the system \eqref{eq:uncimp} is plotted as a function of the minimum dwell-time, for various polynomial degrees. For information, the number of primal/dual variables is 87/27 and 137/33 when the polynomials are of degree 4 and 6, respectively. As suspected, the conservatism is reduced by increasing the degree of the polynomials and the gain decreases as the minimum dwell-time increases.

\begin{figure}
  \centering
  \includegraphics[width=0.75\textwidth]{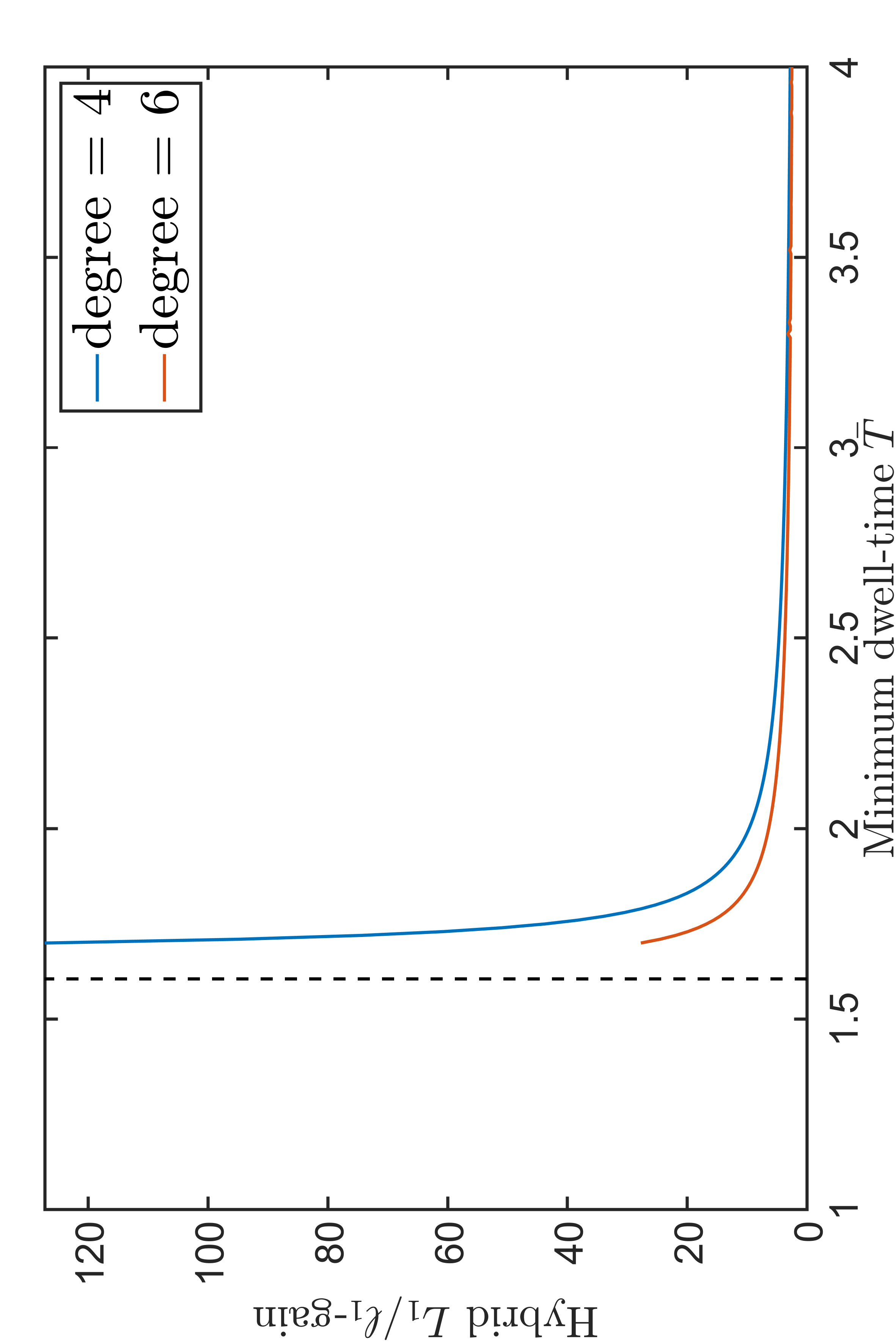}
  \caption{Evolution of the hybrid $L_1/\ell_1$-gain of the system \eqref{eq:uncimp} as a function of the minimum dwell-time and for various polynomial degrees. The vertical black dashed line represents the minimum value 1.67 for the dwell-time in the constant dwell-time case; i.e. periodic impulses.}\label{fig:minDT_L2_impunc}
\end{figure}

\section{Impulsive linear positive systems with delays}\label{sec:lpsd}

We now address the case of impulsive systems with delays. The idea is to first rewrite the time-delay system into an uncertain delay-free system, a method that has proven to be very convenient to work with as all the difficulties are circumvented by ''hiding'' the infinite-dimensional dynamics inside operators that are treated as norm-bounded uncertainties; see e.g. \cite{Zhang:00,RantzerMegretski:97,Kao:07,Briat:book1}. Once rewritten into this form, the results obtained in the previous section become applicable to yield stability and performance conditions for linear positive impulsive systems with delays. There is, however, an additional difficulty here stemming from the fact that the delay operators are not memoryless, which will impose some structural constraints on the set of continuous-time scalings that can be considered. It will be notably shown that the unconstrained scalings case can only arise when a periodicity condition is met by the dwell-time sequence. More specifically, the sequence of dwell-times has to be $h_c/\alpha$-periodic where $\alpha$ is a positive integer.

\subsection{Preliminaries}

Let us consider the following linear system with delay
\begin{equation}\label{eq:posdel}
\begin{array}{rcl}
  \begin{bmatrix}
    \dot{x}(t_k+\tau)\\
     z_c(t_k+\tau)
  \end{bmatrix}&=&\begin{bmatrix}
    A(\tau) & G_c(\tau) & E_c(\tau)\\
    C_c & H_c & F_c
  \end{bmatrix}\begin{bmatrix}
    x(t_k+\tau)\\
    x(t_k+\tau-h_c)\\
    w_c(t_k+\tau)
  \end{bmatrix},\tau\in(0,T_k],k\in\mathbb{Z}_{\ge0}\\
  \begin{bmatrix}
  x(t_k^+)\\
  z_d(k)
  \end{bmatrix}&=&\begin{bmatrix}
    J & G_d & E_d\\
    C_d & H_d & F_d
  \end{bmatrix}\begin{bmatrix}
    x(t_k)\\
    x(t_{k-h_d})\\
    w_d(k)
  \end{bmatrix},\ k\in\mathbb{Z}_{\ge 1}\\
  x(s)&=&\phi_0(s),\ s\in[-h_c,0]
\end{array}
\end{equation}
where $x\in\mathbb{R}^n$, $\phi_0\in C([-h_c,0],\mathbb{R}^n)$, $w_c\in\mathbb{R}^{p_c}$, $w_d\in\mathbb{R}^{p_d}$, $z_c\in\mathbb{R}^{q_c}$ and $z_d\in\mathbb{R}^{q_d}$ are the state of the system, the functional initial condition, the continuous-time exogenous input, the discrete-time exogenous input, the continuous-time performance output and the discrete-time performance output, respectively. The delays $h_c$ and $h_d$ are assumed to be constant. It is well-known that the above system can be rewritten in the form \eqref{eq:mainsyst2} with the same matrices and the additional uncertainty channels
\begin{equation}
  w_{c,\Delta}(t)=x(t-h_c)=(\Delta_cz_{c,\Delta})(t)
\end{equation}
and
\begin{equation}
  w_{d,\Delta}(k)=x(k-h_d)=(\Delta_dz_{d,\Delta})(k)
\end{equation}
together with the identities $z_{c,\Delta}(t)=x(t)$ and $z_{d,\Delta}(k)=x(t_k)$ (i.e. $C_{c,\Delta}=C_{d,\Delta}=I$, $H_{c,\Delta}=H_{d,\Delta}=0$ and $F_{c,\Delta}=F_{d,\Delta}=0$). Constant delay operators are known to have unit $L_1$-gain; see e.g. \cite{Briat:11h,Briat:16b}.

We have the following result regarding the state positivity of the impulsive system \eqref{eq:mainsyst2}.
\begin{proposition}
The following statements are equivalent:
\begin{enumerate}[(a)]
  \item The system \eqref{eq:mainsyst2} is internally positive, i.e. for any $\phi_0\ge0$, $s\in[-h_c,0]$, $w_c(t)\ge0$ and $w_d(k)\ge0$, we have that $x(t),z_c(t)\ge0$ for all $t\ge0$ and $z_d(k)\ge0$ for all $k\in\mathbb{Z}_{\ge0}$.
  \item The matrix-valued function $A(\tau)$ is Metzler for all $\tau\ge0$, the matrix-valued functions $E_c(\tau)$ and $G_c(\tau)$ are nonnegative for all $\tau\ge0$ and the matrices $J,G_d,E_d,C_c,H_c,F_c,C_d,H_d,F_d$ are nonnegative.
\end{enumerate}
\end{proposition}
\blue{\begin{proof}
  The proof relies on the combination of positivity conditions for both continuous-time and discrete-time systems; see e.g. \cite{Farina:00}.
\end{proof}}

\subsection{Scalings}

While the choice for the diagonal discrete-time scaling matrix $S_d$ is obvious in this case, the fact that $S_c$ depends on the value of the clock/timer variable $\tau$ makes it more complicated. It is proven below that the only moment where we can use a timer-dependent scaling is when the scaling exhibits some periodic behavior:
\begin{proposition}
  Let us define the function $\tilde{S}_c(t)=S_c(t_k+\tau)$ when $t\in(t_k,t_{k+1}]$. Then, the two statements are equivalent:
 \begin{enumerate}[(a)]
   \item the function $\tilde{S}_c$ is $h_c$-periodic;
   \item the equality $\tilde{S}_c\Delta_c\tilde{S}_c^{-1}=\Delta_c$ holds.
 \end{enumerate}
 \end{proposition}
\begin{proof}
  Clearly, we have that $\tilde{S}_c(t)(\Delta_c\tilde{S}_c^{-1})(t)=\tilde{S}_c(t)\tilde{S}_c(t-h_c)^{-1}\Delta_c=\diag_i(\tilde s_c^i(t)/\tilde s_c^i(t-h_c))\Delta_c$ where $\tilde S_c=:\diag_i(\tilde s_c^i)$. Hence, we have that $\tilde{S}_c(t)\tilde{S}_c(t-h_c)^{-1}\Delta_c=\Delta_c$ for all $t\ge0$ if and only if $\tilde{S}_c(t)\tilde{S}_c(t-h_c)^{-1}=I$. This is the case if and only  $\tilde{S}_c(t)=\tilde{S}_c(t-h_c)^{-1}$; i.e. the functions are $h_c$-periodic. This proves the result.
\end{proof}
The above result gives a general result based on a periodicity property of the function $\tilde{S}_c$. The problem is that we ignore the fact that this function consists of the concatenation of the elementary function $S_c$ taken on intervals of different lengths. This function depends on the considered dwell-time sequence $\mathcal{T}:=\{T_k\}_{k\ge0}$. In the constant dwell-time case, we have that $\mathcal{T}\in\mathcal{T}_{\bar T}:=\{\{T_0,T_1,\ldots\}:T_k=k\bar T,k\in\mathbb{Z}_{\ge0}\}$ which includes only one sequence. However, in most of the realistic scenarios, we work with families of dwell-sequences. In particular, the set of sequence satisfying a minimum dwell-time condition is given by
$$\mathcal{T}_{\ge\bar T}:=\{\{T_0,T_1,\ldots\}:T_k\ge\bar T,k\in\mathbb{Z}_{\ge0}\}$$
and the set of sequences satisfying a range dwell-time condition by
$$\mathcal{T}_{\in[\Tmin ,\Tmax ]}:=\{\{T_0,T_1,\ldots\}:T_k\in[[\Tmin ,\Tmax ]],k\in\mathbb{Z}_{\ge0}\}.$$

\begin{proposition}
  The function $\tilde{S}_c$ is $h_c$-periodic if and only if $\mathcal{T}$ is a sequence consisting of any repeating sequence of $q\in\mathbb{Z}_{>0}$ dwell-times such that $T_0+\ldots+T_{q-1}=h_c/\alpha$ for some $\alpha\in\mathbb{Z}_{>0}$.
\end{proposition}
\begin{proof}
  The proof is immediate.
\end{proof}
Interestingly, the above result clearly states that the dwell-times need to be, at most, equal to the delay value with the equality holding in the limiting constant dwell-time case; i.e. $q=1$ and $\alpha=1$. This is an immediate consequence of the fact that the function $\tilde S_c$ is constructed by gluing different functions $S_c$ for different dwell-times values. In this regard, it is not possible to have dwell-time values that are strictly larger than the delay. If this is the case, which is likely in practice, then constant scalings will need to considered.

\blue{\begin{remark}[The case of time-varying delays]
  The above discussion illustrates why it is difficult in general to consider time-varying delays. First of all, note that the $L_1$-gain of the time-varying delay operator is equal to $(1-\mu)^{-1}$ where $\mu<1$ is the maximum rate of change of the delay; i.e. $\dot{h}(t)\le \mu$ almost everywhere. Secondly, when the delay is time-varying the commutation property of the time-dependent scaling is unlikely to hold, unless in some very specific scenarios. Note that the case of piecewise constant delays, which would be the simplest to consider, is automatically excluded because it would violate the fact that the derivative of the delay is bounded by $\mu$. In this regard, constant scalings will need to be considered in this case.
\end{remark}}

\subsection{Range dwell-time}

We first address the range dwell-time case, that is, the case where the dwell-time values $T_k$, $k\in\mathbb{Z}_{\ge0}$, belong to the interval $[\Tmin,\Tmax]$ where $0\le \Tmin\Tmax<\infty$. Stability and performance conditions are stated in the following result:
\begin{theorem}[Range dwell-time]\label{th:rdt_g}
Assume that the system \eqref{eq:posdel} is internally positive and that there exist a differentiable vector-valued function $\zeta:[0,\Tmax ]\mapsto\mathbb{R}^n$, $\zeta(0)>0$, a vector-valued function $\mu_c:[0,\Tmax ]\mapsto\mathbb{R}^{n_{c,\Delta}}$, a vector $\mu_d\in\mathbb{R}^{n_{d,\Delta}}$ and scalars $\epsilon,\gamma>0$ such that the conditions $\diag(\mu_c)\in\mathcal{S}_c$, $\diag(\mu_d)\in\mathcal{S}_d$,
\begin{equation}\label{eq:RangeDT1_2}
    \begin{bmatrix}
      \dot{\zeta}(\tau)\\
      0\\
      0
    \end{bmatrix}^T+\begin{bmatrix}
      \zeta(\tau)\\
      \mu_c(\tau)\\
      \mathds{1}
    \end{bmatrix}^T\begin{bmatrix}
      A(\tau) & G_c(\tau) & E_c(\tau)\\
      I & -I & 0\\
      C_c & H_c & F_c
    \end{bmatrix}\le\begin{bmatrix}
    0\\0\\ \gamma\mathds{1}
    \end{bmatrix}^T
  \end{equation}
  and
    \begin{equation}\label{eq:RangeDT2_2}
    \begin{bmatrix}
      -\zeta(\theta)\\
      0
    \end{bmatrix}^T+\begin{bmatrix}
      \zeta(0)\\
      \mathds{1}
    \end{bmatrix}^T\begin{bmatrix}
      J +G_d & E_d\\
      C_d+H_d & F_d
    \end{bmatrix}\le\begin{bmatrix}
      -\epsilon\mathds{1}^T\\ \gamma\mathds{1}
    \end{bmatrix}^T
  \end{equation}
hold for all $\tau\in[0,\Tmax ]$ and all $\theta\in[\Tmin ,\Tmax ]$. Then, the system \eqref{eq:posdel} is asymptotically stable under the range dwell-time condition $[\Tmin ,\Tmax ]$ for all delays $h_c\in\mathbb{R}_{>0}$ and $h_d\in\mathbb{Z}_{\ge 0}$. Moreover, the mapping $(w_c,w_d)\mapsto(z_c,z_d)$ has a hybrid $L_1/\ell_1$-gain of at most $\gamma$.
\end{theorem}
\begin{proof}
The proof follows from substituting the matrices of the system \eqref{eq:posdel} into the conditions of Theorem \ref{th:rangeDT:general} to get
  \begin{equation}
    \begin{bmatrix}
      \dot{\zeta}(\tau)\\
      0\\
      0
    \end{bmatrix}^T+\begin{bmatrix}
      \zeta(\tau)\\
      \mu_c(\tau)\\
      \mathds{1}
    \end{bmatrix}^T\begin{bmatrix}
      A(\tau) & G_c(\tau) & E_c(\tau)\\
      I & -I & 0\\
      C_c & H_c & F_c
    \end{bmatrix}\le\begin{bmatrix}
    0\\0\\ \gamma\mathds{1}
    \end{bmatrix}^T
  \end{equation}
  and
    \begin{equation}
    \begin{bmatrix}
      -\zeta(\theta)\\
      0\\
      0
    \end{bmatrix}^T+\begin{bmatrix}
      \zeta(0)\\
      \mu_d\\
      \mathds{1}
    \end{bmatrix}^T\begin{bmatrix}
      J & G_d & E_d\\
      I & -I & 0\\
      C_d & H_d & F_d
    \end{bmatrix}\le\begin{bmatrix}
      -\epsilon\mathds{1}^T\\0\\ \gamma\mathds{1}
    \end{bmatrix}^T.
  \end{equation}
  In particular, we get that $$\zeta(0)^TG_d-\mu_d^T+\mathds{1}^TH_d\le0$$ and picking $\mu_d^T=\zeta(0)^TG_d-\mu_d^T+\mathds{1}^TH_d$ yields
   \begin{equation}
    \begin{bmatrix}
      -\zeta(\bar T)\\
      0
    \end{bmatrix}^T+\begin{bmatrix}
      \zeta(0)\\
      \mathds{1}
    \end{bmatrix}^T\begin{bmatrix}
      J +G_d & E_d\\
      C_d +H_d & F_d
    \end{bmatrix}\le\begin{bmatrix}
      -\epsilon\mathds{1}^T\\ \gamma\mathds{1}
    \end{bmatrix}^T
  \end{equation}
  and the result follows.
\end{proof}

%

\blue{The following result is the unconstrained scalings counterpart of Theorem \ref{th:rdt_g}:}
\begin{corollary}[Range dwell-time - Unconstrained scaling case]
  Assume that the system \eqref{eq:posdel} is internally positive, that the sequence of dwell-times is restricted to belong to
 \begin{equation}\label{eq:TRDT}
    \mathcal{T}_{RDT,p}:=\left\{\{T_0,T_1,\ldots\}\left|\begin{array}{c}
    T_{qi+k}=\beta_k\in[\Tmin,\Tmax],\sum_{i=0}^{q-1}T_i=h_c/\alpha,\\
    k\in\{0,\ldots,q-1\},(q,\alpha,i)\in\mathbb{Z}_{>0}^2\times\mathbb{Z}_{\ge0},h_c>0
    \end{array}\right.\right\}
  \end{equation}
    and that there exist a differentiable vector-valued function $\zeta:[0,\Tmax ]\mapsto\mathbb{R}^n$, $\zeta(0)>0$, and scalars $\epsilon,\gamma>0$ such that the conditions
\begin{equation}\label{eq:RangeDT1cor}
    \begin{bmatrix}
      \dot{\zeta}(\tau)\\
      0
    \end{bmatrix}^T+\begin{bmatrix}
      \zeta(\tau)\\
      \mathds{1}
    \end{bmatrix}^T\begin{bmatrix}
      A(\tau) + G_c(\tau) & E_c(\tau)\\
      C_c + H_c & F_c
    \end{bmatrix}\le\begin{bmatrix}
    0\\ \gamma\mathds{1}
    \end{bmatrix}^T
  \end{equation}
  and
    \begin{equation}\label{eq:RangeDT2cor}
    \begin{bmatrix}
      -\zeta(\theta)\\
      0
    \end{bmatrix}^T+\begin{bmatrix}
      \zeta(0)\\
      \mathds{1}
    \end{bmatrix}^T\begin{bmatrix}
      J +G_d & E_d\\
      C_d+H_d & F_d
    \end{bmatrix}\le\begin{bmatrix}
      -\epsilon\mathds{1}^T\\ \gamma\mathds{1}
    \end{bmatrix}^T
  \end{equation}
hold for all $\tau\in[0,\Tmax ]$ and all $\theta\in[\Tmin ,\Tmax ]$. Then, the system \eqref{eq:posdel} is asymptotically stable for all delays $h_c\in\mathbb{R}_{>0}$ and $h_d\in\mathbb{Z}_{\ge 0}$, and for all sequences of dwell-times in $\mathcal{T}_{RDT,p}$. Moreover, the map $(w_c,w_d)\mapsto(z_c,z_d)$ has a hybrid $L_1/\ell_1$-gain of at most $\gamma$.
\end{corollary}
\begin{proof}
This follows from simple substitutions as in the proof of Theorem \ref{th:rangeDT:free}.
\end{proof}

\subsection{Minimum dwell-time}

We now address the minimum dwell-time case, that is, the case where the dwell-time values $T_k$, $k\in\mathbb{Z}_{\ge0}$, belong to the interval $[\bar T,\infty)$ where $0<\bar T$. Stability and performance conditions are stated in the following result:
\begin{theorem}[Minimum dwell-time]\label{th:mdt_g}
Assume that the system \eqref{eq:posdel} is internally positive and that the matrices of the system \eqref{eq:mainsyst2} are such that they remain constant for all values $\tau\ge\bar T$. Assume further that there exist a differentiable vector-valued function $\zeta:[0,\bar T]\mapsto\mathbb{R}^n$, $\zeta(0)>0$, $\zeta(\bar T)>0$, a vector-valued function $\mu_c:[0,\bar T]\mapsto\mathbb{R}^{n_{c,\Delta}}$ such $\mu_c(\tau)=\mu_c(\bar T)$ for all $\tau\ge0$ and $\diag(\mu_c)\in\mathcal{S}_c$, and scalars $\epsilon,\gamma>0$ such that the conditions
\begin{equation}\label{eq:MinDT0}
    \begin{bmatrix}
      \zeta(\bar T)\\
      \mu_c(\bar T)\\
      \mathds{1}
    \end{bmatrix}^T\begin{bmatrix}
      A(\bar T) & G_c(\bar T) & E_c(\bar T)\\
      I & -I & 0\\
      C_c & H_c & F_c
    \end{bmatrix}\le\begin{bmatrix}
    -\epsilon\mathds{1}\\0\\ \gamma\mathds{1}
    \end{bmatrix}^T,
  \end{equation}
\begin{equation}\label{eq:MinDT1}
    \begin{bmatrix}
      \dot{\zeta}(\tau)\\
      0\\
      0
    \end{bmatrix}^T+\begin{bmatrix}
      \zeta(\tau)\\
      \mu_c(\tau)\\
      \mathds{1}
    \end{bmatrix}^T\begin{bmatrix}
      A(\tau) & G_c(\tau) & E_c(\tau)\\
      I & -I & 0\\
      C_c & H_c & F_c
    \end{bmatrix}\le\begin{bmatrix}
    0\\0\\ \gamma\mathds{1}
    \end{bmatrix}^T
  \end{equation}
  and
    \begin{equation}\label{eq:MinDT2}
    \begin{bmatrix}
      -\zeta(\bar T)\\
      0
    \end{bmatrix}^T+\begin{bmatrix}
      \zeta(0)\\
      \mathds{1}
    \end{bmatrix}^T\begin{bmatrix}
      J +G_d & E_d\\
      C_d+H_d & F_d
    \end{bmatrix}\le\begin{bmatrix}
      -\epsilon\mathds{1}^T\\  \gamma\mathds{1}
    \end{bmatrix}^T
  \end{equation}
hold for all $\tau\in[0,\bar T]$. Then, the system \eqref{eq:posdel} is asymptotically stable under the minimum dwell-time condition $\bar T$ for all delays $h_c\in\mathbb{R}_{>0}$ and $h_d\in\mathbb{Z}_{\ge 0}$. Moreover, the mapping $(w_c,w_d)\mapsto(z_c,z_d)$ has a hybrid $L_1/\ell_1$-gain of at most $\gamma$.
\end{theorem}

\blue{The following result is the unconstrained scalings counterpart of Theorem \ref{th:mdt_g}:}
\begin{corollary}[Minimum dwell-time - Unconstrained scalings case]
  Assume that the system \eqref{eq:posdel} is internally positive, that the sequence of dwell-times is restricted to
 \begin{equation}\label{eq:TMDT}
     \mathcal{T}_{MDT,p}:=\left\{\{T_0,T_1,\ldots\}\left|\begin{array}{c}
    T_{qi+k}=\beta_k\ge\bar T,\sum_{i=0}^{q-1}T_i=h_c/\alpha,\\
      k\in\{0,\ldots,q-1\},(q,\alpha,i)\in\mathbb{Z}_{>0}^2\times\mathbb{Z}_{\ge0},h_c>0
    \end{array}\right.\right\}
  \end{equation}
  and that there exist a differentiable vector-valued function $\zeta:[0,\bar T ]\mapsto\mathbb{R}^n$, $\zeta(0)>0$, a vector-valued function $\mu_c:[0,\bar T]\mapsto\mathbb{R}^{n_{c,\Delta}}$, a vector $\mu_d\in\mathbb{R}^{n_{d,\Delta}}$ and scalars $\epsilon,\gamma>0$ such that the conditions
  \begin{equation}
  \begin{bmatrix}
      \zeta(\bar T)\\
      \mathds{1}
    \end{bmatrix}^T\begin{bmatrix}
      A(\bar T) + G_c(\bar T) & E_c(\bar T)\\
      C_c + H_c & F_c
    \end{bmatrix}\le\begin{bmatrix}
    \epsilon \mathds{1}\\ \gamma\mathds{1}
    \end{bmatrix}^T,
  \end{equation}
\begin{equation}
    \begin{bmatrix}
      \dot{\zeta}(\tau)\\
      0
    \end{bmatrix}^T+\begin{bmatrix}
      \zeta(\tau)\\
      \mathds{1}
    \end{bmatrix}^T\begin{bmatrix}
      A(\tau) + G_c(\tau) & E_c(\tau)\\
      C_c + H_c & F_c
    \end{bmatrix}\le\begin{bmatrix}
    0\\ \gamma\mathds{1}
    \end{bmatrix}^T
  \end{equation}
  and
    \begin{equation}
    \begin{bmatrix}
      -\zeta(\theta)\\
      0
    \end{bmatrix}^T+\begin{bmatrix}
      \zeta(0)\\
      \mathds{1}
    \end{bmatrix}^T\begin{bmatrix}
      J +G_d & E_d\\
      C_d+H_d & F_d
    \end{bmatrix}\le\begin{bmatrix}
      -\epsilon\mathds{1}^T\\ \gamma\mathds{1}
    \end{bmatrix}^T
  \end{equation}
hold for all $\tau\in[0,\bar T]$. Then, the system \eqref{eq:posdel} is asymptotically stable for all delays $h_c\in\mathbb{R}_{>0}$ and $h_d\in\mathbb{Z}_{\ge 0}$, and for all sequences of dwell-times in $\mathcal{T}_{MDT,p}$. Moreover, the mapping $(w_c,w_d)\mapsto(z_c,z_d)$ has a hybrid $L_1/\ell_1$-gain of at most $\gamma$.
\end{corollary}


\section{Interval observation of linear impulsive systems with delays}\label{sec:IIimp}

Let us consider here the following class of linear impulsive systems with delays
\begin{equation}\label{eq:linimpdel}
\begin{array}{rcl}
  \dot{x}(t)&=&Ax(t)+G_cx(t-h_c)+E_cw_c(t),\ t\ne t_k\\
  x(t_k^+)&=&Jx(t_k)+G_dx(t_{k-h_d})+E_dw_d(k),\ k\in\mathbb{Z}_{\ge 1}\\
  y_c(t)&=&C_{yc}x(t)+H_{yc}x(t-h_c)+F_{yc}w_c(t)\\
  y_d(k)&=&C_{yd}x(t_k)+H_{yd}x(t_{k-h_d})+F_{yd}w_d(k)\\
  %
x(s)&=&\phi_0(s),s\in[-h_c,0]
\end{array}
\end{equation}
where$x\in\mathbb{R}^n$, $\phi_0\in C([-h_c,0],\mathbb{R}^n)$, $w_c\in\mathbb{R}^{p_c}$, $w_d\in\mathbb{R}^{p_d}$, $y_c\in\mathbb{R}^{r_c}$ and $y_d\in\mathbb{R}^{r_d}$ are the state of the system, the functional initial condition, the continuous-time exogenous input, the discrete-time exogenous input, the continuous-time measured output and the discrete-time measured output, respectively. The input signals are all assumed to be bounded functions and that some bounds are known; i.e. we have $w_c^-(t)\le w_c(t)\le w_c^+(t)$ and $w_d^-(k)\le w_d(k)\le w_d^+(k)$ for all $t\ge0$ and $k\ge0$ and for some known $w_c^-(t), w_c^+(t),w_d^-(k),w_d^+(k)$.

\subsection{Proposed interval observer}

We are interested in finding an interval-observer of the form
\begin{equation}\label{eq:obs}
\begin{array}{lcl}
      \dot{x}^\bullet(t)&=&Ax^\bullet(t)+G_cx^\bullet(t-h_c)+E_c w_c^\bullet(t)\\
      &&+L_c(t)(y_c(t)-C_{yc} x^\bullet(t)-H_{yc}x^\bullet(t-h_c)-F_{yc} w_c^\bullet(t))\\
      x^\bullet(t_k^+)&=&Jx^\bullet(t_k)+G_dx^\bullet(t_{k-h_d})+E_d w_d^\bullet(t)\\
&&\quad+L_d(y_d(k)-C_{yd} x^\bullet(t_k)-H_{yd} x^\bullet(t_{k-h_d})-F_{yd} w_d^\bullet(t))\\
      x^\bullet(s)&=&\phi_0^\bullet(s),s\in[-h_c,0]
\end{array}
\end{equation}
where $\bullet\in\{-,+\}$. Above, the observer with the superscript ``$+$'' is meant to estimate an upper-bound on the state value whereas the observer with the superscript ``-'' is meant to estimate a lower-bound, i.e. $x^-(t)\le x(t)\le x^+(t)$ for all $t\ge0$ provided that $\phi_0^-\le \phi_0\le \phi_0^+$, $w_c^-(t)\le w_c(t)\le w_c^+(t)$ and $w_d^-(k)\le w_d(k)\le w_d^+(k)$. The errors dynamics $e^+(t):=x^+(t)-x(t)$ and $e^-(t):=x(t)-x^-(t)$ are then described by
\begin{equation}\label{eq:error}
\begin{array}{rcl}
    \dot{e}^\bullet(t)&=&(A-L_c(t)C_{yc})e^\bullet(t)+(G_c-L_c(t)C_{yc})e^\bullet(t-h_c)+(E_c-L_c(t) F_{yc})\delta_c^\bullet(t)\\
    e^\bullet(t_k^+)&=&(J-L_d C_{yd})e^\bullet(t_k)+(G_d-L_d H_{yd})e^\bullet(t_{k-h_d})+(E_d-L_d F_{yd})\delta_d^\bullet(k)\\
    e_c^\bullet(t)&=&M_ce^\bullet(t)\\
    e_d^\bullet(k)&=&M_de^\bullet(t_k)\\
    e^\bullet(s)&=&\phi_{e,0}^\bullet(s),s\in[-h_c,0]
\end{array}
\end{equation}
where $\bullet\in\{-,+\}$, $\delta_c^+(t):=w_c^+(t)-w_c(t)\in\mathbb{R}_{\ge0}^{p_c}$, $\delta_c^-(t):=w_c(t)-w_c^-(t)\in\mathbb{R}_{\ge0}^{p_c}$, $\delta_d^+(k):=w_d^+(k)-w_d(k)\in\mathbb{R}_{\ge0}^{p_d}$ and $\delta_d^-(k):=w_d(k)-w_d^-(k)\in\mathbb{R}_{\ge0}^{p_d}$. The continuous-time and discrete-time performance outputs are denoted by  $e_c^\bullet(t)$ and $e_d^\bullet(k)$, respectively. The initial conditions are defined as $\phi_{e,0}^+:=\phi_0^+-\phi_0$ and $\phi_{e,0}^-:=\phi_0-\phi_0^-$. Note that both errors have exactly the same dynamics and, consequently, it is unnecessary here to consider different observer gains. Note that this would not be the case if the observers were coupled in a non-symmetric way. The matrices $M_c,M_d\in\mathbb{R}^{n\times n}_{\ge0}$ are nonzero weighting matrices that are needed to be chosen a priori.

\subsection{Range dwell-time result}

In the range-dwell -time case, the time-varying gain $L_c(t)$ in \eqref{eq:obs} is defined as follows
\begin{equation}\label{eq:L1}
  L_c(t)=\tilde{L}_c(t-t_k),\ t\in(t_k,t_{k+1}]
\end{equation}
where $\tilde{L}_c:[0,\Tmax ]\mapsto\mathbb{R}^{n\times q_c}$ is a matrix-valued function to be determined. The rationale for considering such structure is to allow for the derivation of convex design conditions. The observation problem is defined, in this case, as follows:
\begin{problem}\label{problem1}
Find an interval observer of the form \eqref{eq:obs} (i.e. a matrix-valued function $L_c(\cdot)$ of the form \eqref{eq:L1} and a matrix $L_d\in\mathbb{R}^{n\times q_d}$) such that the error dynamics \eqref{eq:error} is
  \begin{enumerate}[(a)]
    \item state-positive, that is
    \begin{itemize}
      \item $A-\tilde L_c(\tau)  C_c$ is Metzler for all $\tau\in[0,\Tmax ]$,
      \item $G_c-\tilde L_c(t)C_{yc}$ and $E_c-\tilde L_c(t) F_{yc}$ are nonnegative for all $\tau\in[0,\Tmax ]$,
      \item $J-L_dC_d$, $E_d-L_dF_d$ and $G_d-L_d H_{yd}$ are nonnegative,
    \end{itemize}
    \item asymptotically stable under range dwell-time $[\Tmin ,\Tmax ]$ when $w_c\equiv0$ and $w_d\equiv0$, and
    \item the map
    \begin{equation}
      (\delta_c^\bullet,\delta_d^\bullet)\mapsto(e_c^\bullet,e_d^\bullet)
    \end{equation}
    has a hybrid $L_1/\ell_1$-gain of at most $\gamma$.
  \end{enumerate}
\end{problem}

The following result provides a sufficient condition for the solvability of Problem \ref{problem1}:
\begin{theorem}\label{th:1}\label{th:RDTinterval}
Assume that there exist a differentiable matrix-valued function $X:[0,\Tmax ]\mapsto\mathbb{D}^n$, $X(0)\succ0$, a matrix-valued function $Y_c:[0,\Tmax ]\mapsto\mathbb{R}^{n\times q_c}$, matrices $Y_d\in\mathbb{R}^{n\times q_d}$, $U_c\in\mathbb{D}^{n}_{\succ0}$ and scalars $\eps,\alpha,\gamma>0$ such that the conditions
\begin{subequations}\label{eq:th1a}
\begin{alignat}{4}
            X(\tau)A-Y_c(\tau)C_{yc}+\alpha I_n&\ge0\label{eq:th1:1}\\
            X(\tau)G_c-Y_c(\tau)H_{yc}&\ge0\label{eq:th1:2}\\
            X(\tau)E_c-Y_c(\tau)F_c&\ge0\label{eq:th1:3}\\
            X(0) J-Y_d C_{yd}&\ge0\label{eq:th1:4}\\
            X(0) G_d-Y_d H_{yd}&\ge0\label{eq:th1:5}\\
            X(0)E_d-Y_d F_d&\ge0\label{eq:th1:6}
  \end{alignat}
\end{subequations}
           and
 \begin{subequations}\label{eq:th1b}
\begin{alignat}{4}
           \mathds{1}^T\left[\dot{X}(\tau)+X(\tau)A-Y_c(\tau)C_{yc}+U_c\right]+\mathds{1}^TM_c&\le0\label{eq:th1:5}\\
           \mathds{1}^T\left[X(\tau)G_c-Y_c(\tau)H_{yc}-U_c\right]&\le0\label{eq:th1:5}\\
            \mathds{1}^T\left[X(\tau)E_c-Y_c(\tau)F_c\right]-\gamma \mathds{1}^T&\le0\label{eq:th1:7}\\
            \mathds{1}^T\left[X(0)(J+G_d)-Y_d(C_{yd}+H_{yd})-X(\theta)+X(0)+\eps I\right]+\mathds{1}^TM_d&\le0\label{eq:th1:6}\\
            \mathds{1}^T\left[X(0)E_d-Y_d F_d\right]-\gamma \mathds{1}^T&\le0\label{eq:th1:8}
             \end{alignat}
\end{subequations}
%
%
        %
hold for all $\tau\in[0,\Tmax ]$ and all $\theta\in[\Tmin ,\Tmax ]$. Then, there exists an interval observer of the form \eqref{eq:obs}-\eqref{eq:L1} that solves Problem \ref{problem1} and suitable observer gains are given by
\begin{equation}\label{eq:formula1}
  \tilde{L}_c(\tau)= X(\tau)^{-1}Y_c(\tau)\quad \textnormal{and}\quad L_d=X(0)^{-1}Y_d.
\end{equation}
\end{theorem}
\blue{\begin{proof}
  First note that with the changes of variables $\zeta(\tau)^T=:\mathds{1}^TX(\tau)$, the conditions \eqref{eq:th1a} are exactly the positivity conditions for the error dynamics \eqref{eq:error}. Secondly, the conditions \eqref{eq:th1b} coincide with the stability conditions of Theorem \ref{th:rangeDT:general} with the changes of variables $Y_c(\tau)=X(\tau)^{-1}\tilde{L}_c(\tau)$ and $Y_d=X(0)^{-1}L_d$. This proves the result.
\end{proof}}
\blue{The following result is the unconstrained scalings counterpart of Theorem \ref{th:RDTinterval}:}
\begin{corollary}[Range dwell-time - Unconstrained scalings case]\label{th:RDTinterval_uncon}
Assume that there exist a differentiable matrix-valued function $X:[0,\Tmax ]\mapsto\mathbb{D}^n$, $X(0)\succ0$, a matrix-valued function $Y_c:[0,\Tmax ]\mapsto\mathbb{R}^{n\times q_c}$, matrices $Y_d\in\mathbb{R}^{n\times q_d}$, $U_c\in\mathbb{D}^{n}_{\succ0}$ and scalars $\eps,\alpha,\gamma>0$ such that the conditions
\begin{subequations}
\begin{alignat}{4}
            X(\tau)A-Y_c(\tau)C_{yc}+\alpha I_n&\ge0\\
            X(\tau)G_c-Y_c(\tau)H_{yc}&\ge0\\
            X(\tau)E_c-Y_c(\tau)F_c&\ge0\\
            X(0) J-Y_d C_{yd}&\ge0\\
            X(0) G_d-Y_d H_{yd}&\ge0\\
            X(0)E_d-Y_d F_d&\ge0
  \end{alignat}
\end{subequations}
           and
 \begin{subequations}
\begin{alignat}{4}
           \mathds{1}^T\left[\dot{X}(\tau)+X(\tau)(A+G_c)-Y_c(\tau)(C_{yc}+H_{yc})\right]+\mathds{1}^TM_c&\le0\\
            \mathds{1}^T\left[X(\tau)E_c-Y_c(\tau)F_c\right]-\gamma \mathds{1}^T&\le0\label{eq:th2:7}\\
            \mathds{1}^T\left[X(0)(J+G_d)-Y_d(C_{yd}+ H_{yd})-X(\theta)+\eps I\right]+\mathds{1}^TM_d&\le0\\
            \mathds{1}^T\left[X(0)E_d-Y_d F_d\right]-\gamma \mathds{1}^T&\le0
             \end{alignat}
\end{subequations}
hold for all $\tau\in[0,\Tmax ]$ and $\theta\in[\Tmin ,\Tmax ]$. Then, there exists an interval observer of the form \eqref{eq:obs}-\eqref{eq:L2} that solves Problem \ref{problem2} with the additional restriction that sequence of dwell-times belongs to $\mathcal{T}_{RDT,p}$ defined in \eqref{eq:TRDT}. Moreover, suitable observer gains are given by
\begin{equation}\label{eq:formula2}
  \tilde{L}_c(\tau)= X(\tau)^{-1}Y_c(\tau)\quad \textnormal{and}\quad L_d=X(0)^{-1}Y_d.
\end{equation}
\end{corollary}
\blue{\begin{proof}
The proof follows from the same lines at the proof of Theorem \ref{th:RDTinterval} with the difference that Theorem \ref{th:rangeDT:free} is considered as the stability result.
\end{proof}}

\subsection{Minimum dwell-time result}

In the minimum dwell-time case, the time-varying gain  $L_c$ is defined as follows
\begin{equation}\label{eq:L2}
  L_c(t)=\left\{\begin{array}{ll}
  \tilde{L}_c(t-t_k)& \textnormal{if }t\in(t_k,t_{k}+\tau]\\
  \tilde{L}_c(\bar T)& \textnormal{if }t\in(t_k+\bar{T},t_{k+1}]
  \end{array}\right.
\end{equation}
where $\tilde{L}_c:\mathbb{R}_{\ge0}\mapsto\mathbb{R}^{n\times q_c}$ is a function to be determined. As in the range dwell-time case, the structure is chosen to facilitate the derivation of convex design conditions. The observation problem is defined, in this case, as follows:
\begin{problem}\label{problem2}
Find an interval observer of the form \eqref{eq:obs} (i.e. a matrix-valued function $L_c(\cdot)$ of the form \eqref{eq:L2} and a matrix $L_d\in\mathbb{R}^{n\times q_d}$) such that the error dynamics \eqref{eq:error} is
  \begin{enumerate}[(a)]
    \item state-positive, that is
    \begin{itemize}
      \item $A-\tilde L_c(\tau)  C_c$ is Metzler for all $\tau\in[0,\Tmax ]$,
      \item $G_c-\tilde L_c(t)C_{yc}$ and $E_c-\tilde L_c(t) F_{yc}$ are nonnegative for all $\tau\in[0,\Tmax ]$,
      \item $J-L_dC_d$, $E_d-L_dF_d$ and $G_d-L_d H_{yd}$ are nonnegative,
    \end{itemize}
    \item asymptotically stable under minimum dwell-time $\bar T$ when $w_c\equiv0$ and $w_d\equiv0$, and
     \item the map
    \begin{equation}
      (\delta_c^\bullet,\delta_d^\bullet)\mapsto(e_c^\bullet,e_d^\bullet)
    \end{equation}
    has a hybrid $L_1/\ell_1$-gain of at most $\gamma$.
  \end{enumerate}
\end{problem}

The following result provides a sufficient condition for the solvability of Problem \ref{problem2}:
\begin{theorem}\label{th:2}\label{th:MDTinterval}
Assume that there exist a differentiable matrix-valued function $X:[0,\bar T ]\mapsto\mathbb{D}^n$, $X(0)\succ0$, a matrix-valued function $Y_c:[0,\bar T ]\mapsto\mathbb{R}^{n\times q_c}$, matrices $Y_d\in\mathbb{R}^{n\times q_d}$, $U_c\in\mathbb{D}^{n}_{\succ0}$ and scalars $\eps,\alpha,\gamma>0$ such that the conditions
\begin{subequations}\label{eq:th2a}
\begin{alignat}{4}
            X(\tau)A-Y_c(\tau)C_{yc}+\alpha I_n&\ge0\label{eq:th2:1}\\
            X(\tau)G_c-Y_c(\tau)H_{yc}&\ge0\label{eq:th2:2}\\
            X(\tau)E_c-Y_c(\tau)F_c&\ge0\label{eq:th2:3}\\
            X(0) J-Y_d C_{yd}&\ge0\label{eq:th2:4}\\
            X(0) G_d-Y_d H_{yd}&\ge0\label{eq:th2:5}\\
            X(0)E_d-Y_d F_d&\ge0\label{eq:th2:6}
  \end{alignat}
\end{subequations}
           and
 \begin{subequations}\label{eq:th2b}
\begin{alignat}{4}
        \mathds{1}^T\left[X(\bar T)A-Y_c(\bar T)C_{yc}+U_c\right]+\mathds{1}^TM_c&\le0\label{eq:th2:50}\\
           \mathds{1}^T\left[X(\bar T)G_c-Y_c(\bar T)H_{yc}-U_c\right]&\le0\label{eq:th2:50}\\
            \mathds{1}^T\left[X(\bar T)E_c-Y_c(\bar T)F_c\right]-\gamma \mathds{1}^T&\le0\label{eq:th2:70}\\
           \mathds{1}^T\left[\dot{X}(\tau)+X(\tau)A-Y_c(\tau)C_{yc}+U_c\right]+\mathds{1}^TM_c&\le0\label{eq:th2:5}\\
           \mathds{1}^T\left[X(\tau)G_c-Y_c(\tau)H_{yc}-U_c\right]&\le0\label{eq:th2:5}\\
            \mathds{1}^T\left[X(\tau)E_c-Y_c(\tau)F_c\right]-\gamma \mathds{1}^T&\le0\label{eq:th2:7}\\
            \mathds{1}^T\left[X(0)(J+G_d)-Y_d(C_{yd}+H_d)-X(\bar T)+\eps I\right]+\mathds{1}^TM_d&\le0\label{eq:th2:6}\\
            \mathds{1}^T\left[X(0)E_d-Y_d F_d\right]-\gamma \mathds{1}^T&\le0\label{eq:th2:8}
             \end{alignat}
\end{subequations}
%
%
        %
hold for all $\tau\in[0,\bar T]$. Then, there exists an interval observer of the form \eqref{eq:obs}-\eqref{eq:L2} that solves Problem \ref{problem2} and suitable observer gains are given by
\begin{equation}\label{eq:formula2}
  \tilde{L}_c(\tau)= X(\tau)^{-1}Y_c(\tau)\quad \textnormal{and}\quad L_d=X(0)^{-1}Y_d.
\end{equation}
\end{theorem}
\blue{\begin{proof}
The proof follows from the same lines at the proof of Theorem \ref{th:RDTinterval} with the difference that Theorem \ref{th:minimumDT:general} is considered as the stability result.
\end{proof}}
\blue{The following result is the unconstrained scalings counterpart of Theorem \ref{th:MDTinterval}:}
\begin{corollary}[Minimum dwell-time - Unconstrained scalings case]\label{th:MDTinterval_unc}
Assume that there exist a differentiable matrix-valued function $X:[0,\bar T ]\mapsto\mathbb{D}^n$, $X(0)\succ0$, a matrix-valued function $Y_c:[0,\bar T]\mapsto\mathbb{R}^{n\times q_c}$, matrices $Y_d\in\mathbb{R}^{n\times q_d}$, $U_c\in\mathbb{D}^{n}_{\succ0}$ and scalars $\eps,\alpha,\gamma>0$ such that the conditions
\begin{subequations}
\begin{alignat}{4}
            X(\tau)A-Y_c(\tau)C_{yc}+\alpha I_n&\ge0\\
            X(\tau)G_c-Y_c(\tau)H_{yc}&\ge0\\
            X(\tau)E_c-Y_c(\tau)F_c&\ge0\\
            X(0) J-Y_d C_{yd}&\ge0\\
            X(0) G_d-Y_d H_{yd}&\ge0\\
            X(0)E_d-Y_d F_d&\ge0
  \end{alignat}
\end{subequations}
           and
 \begin{subequations}
\begin{alignat}{4}
        \mathds{1}^T\left[X(\bar T)(A+G_c)-Y_c(\bar T)(C_{yc}+H_{yc})\right]+\mathds{1}^TM_c&\le0\\
            \mathds{1}^T\left[X(\bar T)E_c-Y_c(\bar T)F_c\right]-\gamma \mathds{1}^T&\le0\\
           \mathds{1}^T\left[\dot{X}(\tau)+X(\tau)(A+G_c)-Y_c(\tau)(C_{yc}+H_{yc})\right]+\mathds{1}^TM_c&\le0\\
            \mathds{1}^T\left[X(\tau)E_c-Y_c(\tau)F_c\right]-\gamma \mathds{1}^T&\le0\label{eq:th2:7}\\
            \mathds{1}^T\left[X(0)(J+G_d)-Y_d(C_{yd}+ H_{yd})-X(\bar T)+\eps I\right]+\mathds{1}^TM_d&\le0\\
            \mathds{1}^T\left[X(0)E_d-Y_d F_d\right]-\gamma \mathds{1}^T&\le0
             \end{alignat}
\end{subequations}
hold for all $\tau\in[0,\bar T]$. Then, there exists an interval observer of the form \eqref{eq:obs}-\eqref{eq:L2} that solves Problem \ref{problem2} with the additional restriction that sequence of dwell-times belongs to $\mathcal{T}_{MDT,p}$ defined in \eqref{eq:TMDT}.
Moreover, suitable observer gains are given by
\begin{equation}\label{eq:formula2}
  \tilde{L}_c(\tau)= X(\tau)^{-1}Y_c(\tau)\quad \textnormal{and}\quad L_d=X(0)^{-1}Y_d.
\end{equation}
\end{corollary}
\blue{\begin{proof}
The proof follows from the same lines at the proof of Theorem \ref{th:RDTinterval} with the difference that Theorem \ref{th:minimumDT:free} is considered as the stability result.
\end{proof}}

\subsection{Examples}

All the computations are performed on a computer equipped with a processor i7-5600U@2.60GHz with 16GB of RAM. The conditions are implemented using SOSTOOLS \cite{sostools3} and solved with SeDuMi \cite{Sturm:01a}.

\subsubsection{Range dwell-time}

Let us consider now the system \eqref{eq:linimpdel} with the matrices
\begin{equation}\label{eq:ex2}
\begin{array}{l}
    A=\begin{bmatrix}
    -2 & 1\\
    0 & 1
  \end{bmatrix}, G_c=\begin{bmatrix}
    0.5 & 0.1\\0 & 0.1
      \end{bmatrix}, E_c=\begin{bmatrix}
    0.1\\
    0.1
  \end{bmatrix},\\ J=\begin{bmatrix}
1.1 & 0\\
0 & 0.1
  \end{bmatrix},  G_d=\begin{bmatrix}
    0.1 & 0\\0 & 0.1
  \end{bmatrix}, E_d=\begin{bmatrix}
    0.3\\
    0.3
  \end{bmatrix},\\
    C_{yc}=C_{yd}=\begin{bmatrix}
    0 & 1
  \end{bmatrix}, H_{yc}=H_{yd}=\begin{bmatrix}
    0 & 0
  \end{bmatrix}, F_{yc}=F_{yd}=0.1.
\end{array}
\end{equation}
Define also $w_c(t)=4\sin(t)$, $w^-(t)=-4$, $w^+(t)=4$, $w_d(k)$ is a stationary random process that follows the uniform distribution $\mathcal{U}(-1,1)$, $w_d^-=-1$ and $w_d^+=1$. Using polynomials of degree 4 and a constant scaling $\mu_c$ together with $T_{min}=0.3$ and $T_{max}=0.5$ in Theorem \ref{th:RDTinterval} yields the minimum $\gamma=2.33$. The optimization problem has 478 and 136 primal and dual variables, and is solved in 3.28 seconds. The computed gains are
\begin{equation}\label{eq:Lex1b}
  L_d=\begin{bmatrix}
     0\\
     0.1
  \end{bmatrix}\quad \textnormal{and}\quad\tilde L_c(\tau)=\begin{bmatrix}
    \dfrac{0.7545\tau^4   -2.8460 \tau^3+   9.4306\tau^2+    7.2310\tau+    7.7262}{0.2455 \tau^4  -3.3206\tau^3+   10.0521\tau^2+    7.0819\tau+    7.7364}\\
    \dfrac{1.5068 \tau^4  -1.6153\tau^3+    0.5277\tau^2   -7.3862 \tau+   7.9794}{1.5068\tau^4   -1.6153\tau^3+    0.5277\tau^2   -7.3862 \tau+   7.9794}\\
  \end{bmatrix}.
\end{equation}
Note that the gain $L_c$ is constant and has been obtained from an approximation of the $\tau$-dependent gain which deviates from a very small amount from the above value. To illustrate this result, we generate random impulse times satisfying the range dwell-time condition and we obtain the trajectories depicted in Fig.~\ref{fig:states_rangeDTZ}. The disturbance inputs are depicted in Fig.~\ref{fig:inputs_rangeDTZ}. For simulation purposes, we set $h_c=2$ and $h_d=4$.

\begin{figure}
  \centering
  \includegraphics[width=0.75\textwidth]{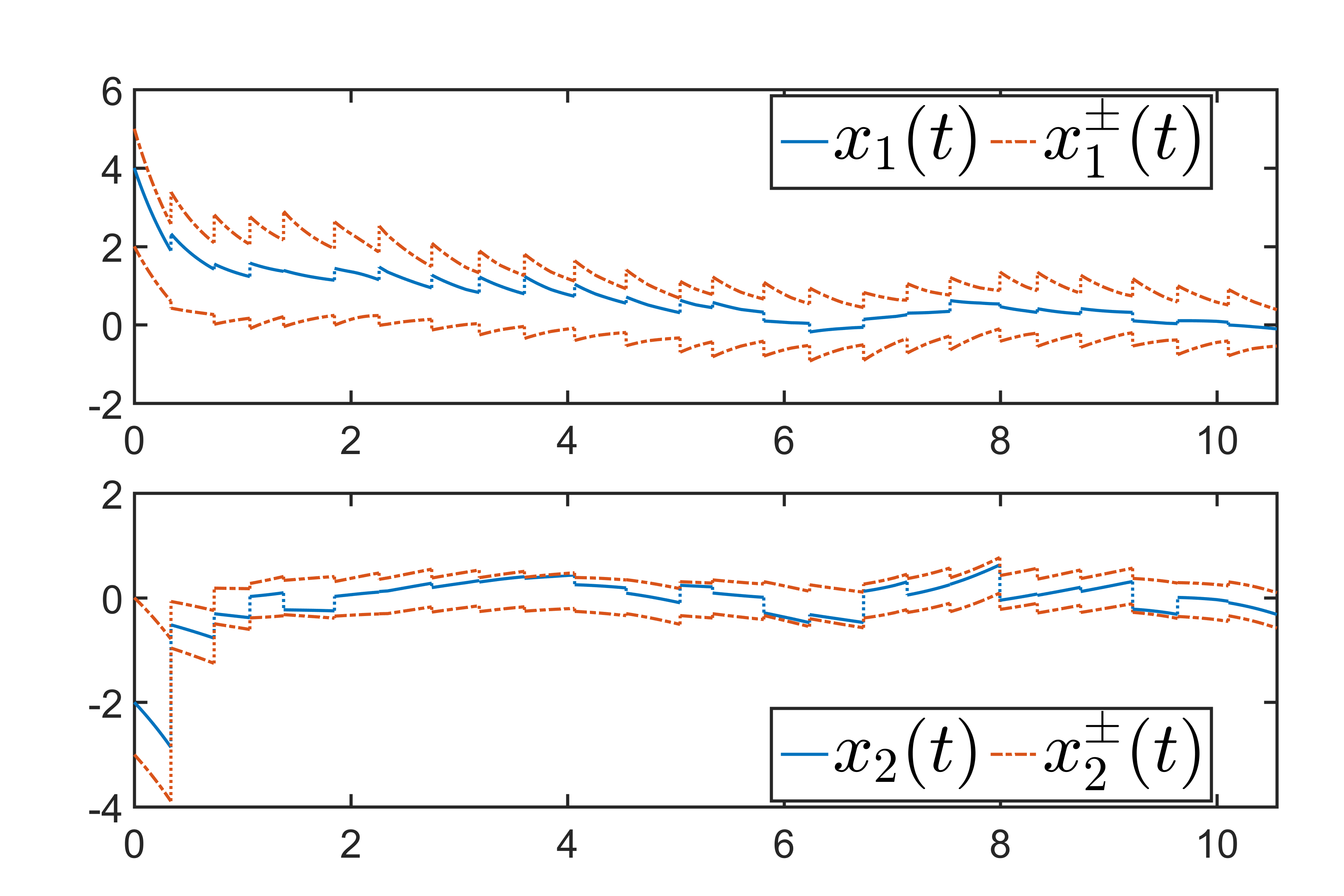}
  \caption{\textbf{Constant scaling $\mu_c$.} Trajectories of the system \eqref{eq:linimpdel}-\eqref{eq:ex2} and the interval observer \eqref{eq:obs} for some randomly chosen impulse times satisfying the range dwell-time $[0.3,\ 0.5]$.}\label{fig:states_rangeDTZ}
\end{figure}

\begin{figure}
  \centering
  \includegraphics[width=0.75\textwidth]{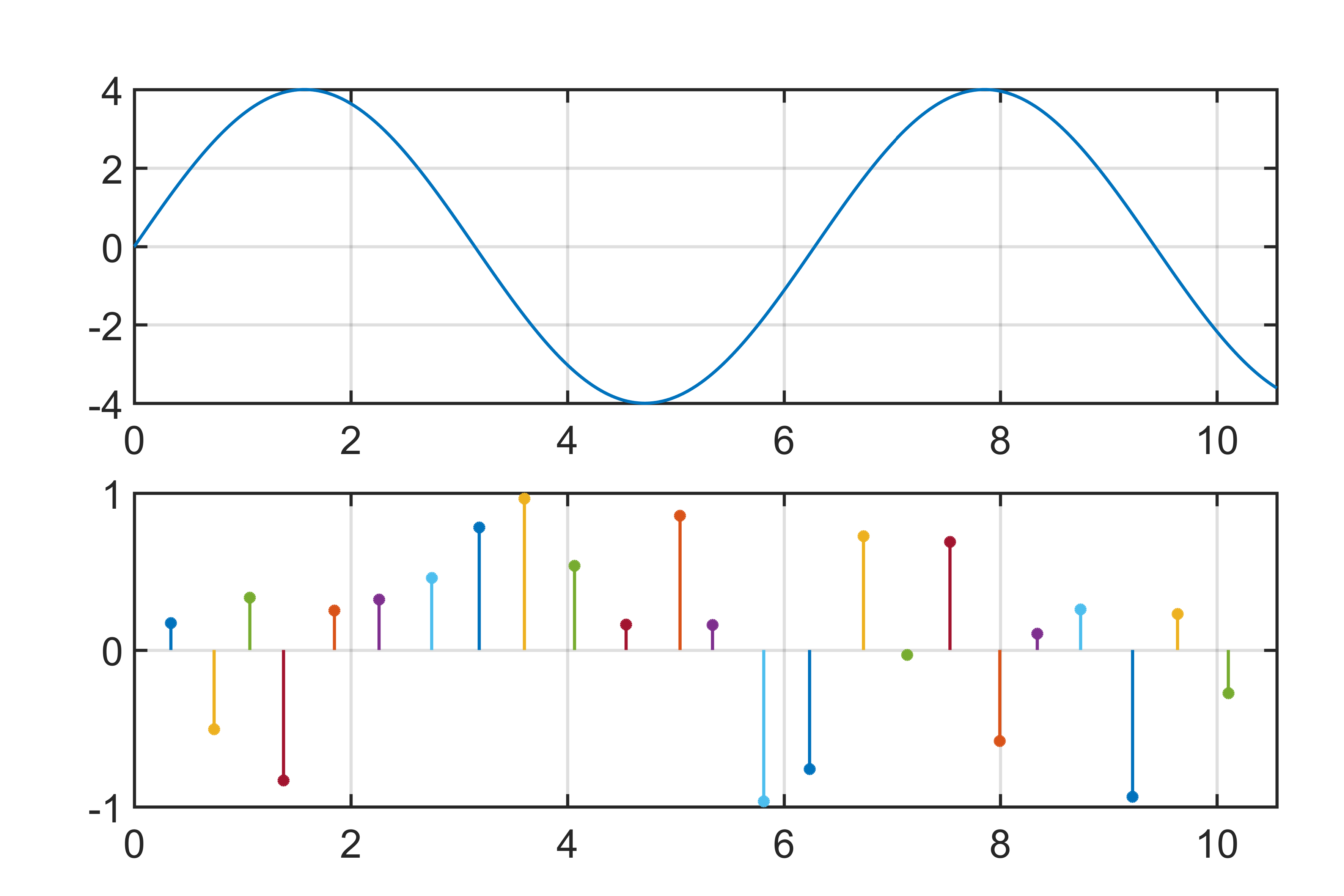}
  \caption{\textbf{Constant scaling $\mu_c$.} Trajectory of the continuous-time input $w_c$ (top) and the discrete-time input $w_d$ (bottom)}\label{fig:inputs_rangeDTZ}
\end{figure}

We now consider the Corollary \ref{th:RDTinterval_uncon} and we get the minimum value  $\gamma=1.7191$. The problem has 486 primal variables, 136 dual variables and it takes 2.90 seconds to solve. The following observer gains are obtained
\begin{equation}\label{eq:Lex1}
  L_d=\begin{bmatrix}
     0\\
     0.1
  \end{bmatrix}\quad \textnormal{and}\quad \tilde L_c(\tau)=\begin{bmatrix}
    \dfrac{ 0.6771\tau^4+0.6706\tau^3+1.2297\tau^2+1.8980\tau+1.4996}{0.3229\tau^4+0.6787\tau^3+1.3914\tau^2+1.8528\tau+1.5029}\\
    \dfrac{5.6213\tau^4-3.9871\tau^3-0.6884\tau^2-6.6036\tau+8.7950}{2.4898\tau^4   -3.3851\tau^3+    0.3116\tau^2   -6.9083\tau+    8.8177}\\
  \end{bmatrix}.
\end{equation}
Simulation results are depicted in Fig. \ref{fig:states_rangeDT} and Fig. \ref{fig:inputs_rangeDT}. Notice the periodicity of the sequence of dwell-times.
\begin{figure}
  \centering
  \includegraphics[width=0.75\textwidth]{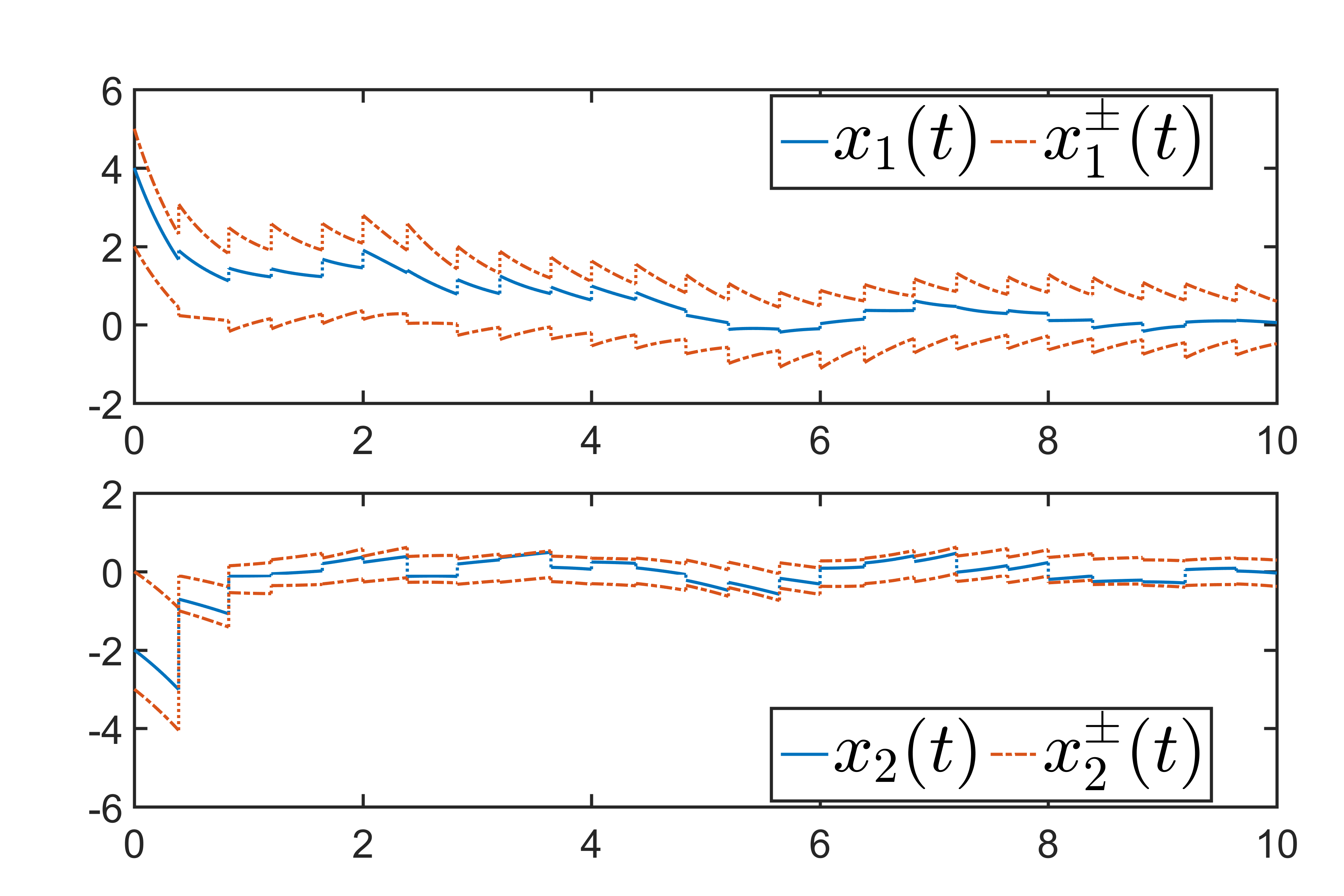}
  \caption{\textbf{Unconstrained scaling $\mu_c$.} Trajectories of the system \eqref{eq:linimpdel}-\eqref{eq:ex2} and the interval observer \eqref{eq:obs} for some randomly chosen impulse times satisfying the range dwell-time $[0.3,\ 0.5]$.}\label{fig:states_rangeDT}
\end{figure}

\begin{figure}
  \centering
  \includegraphics[width=0.75\textwidth]{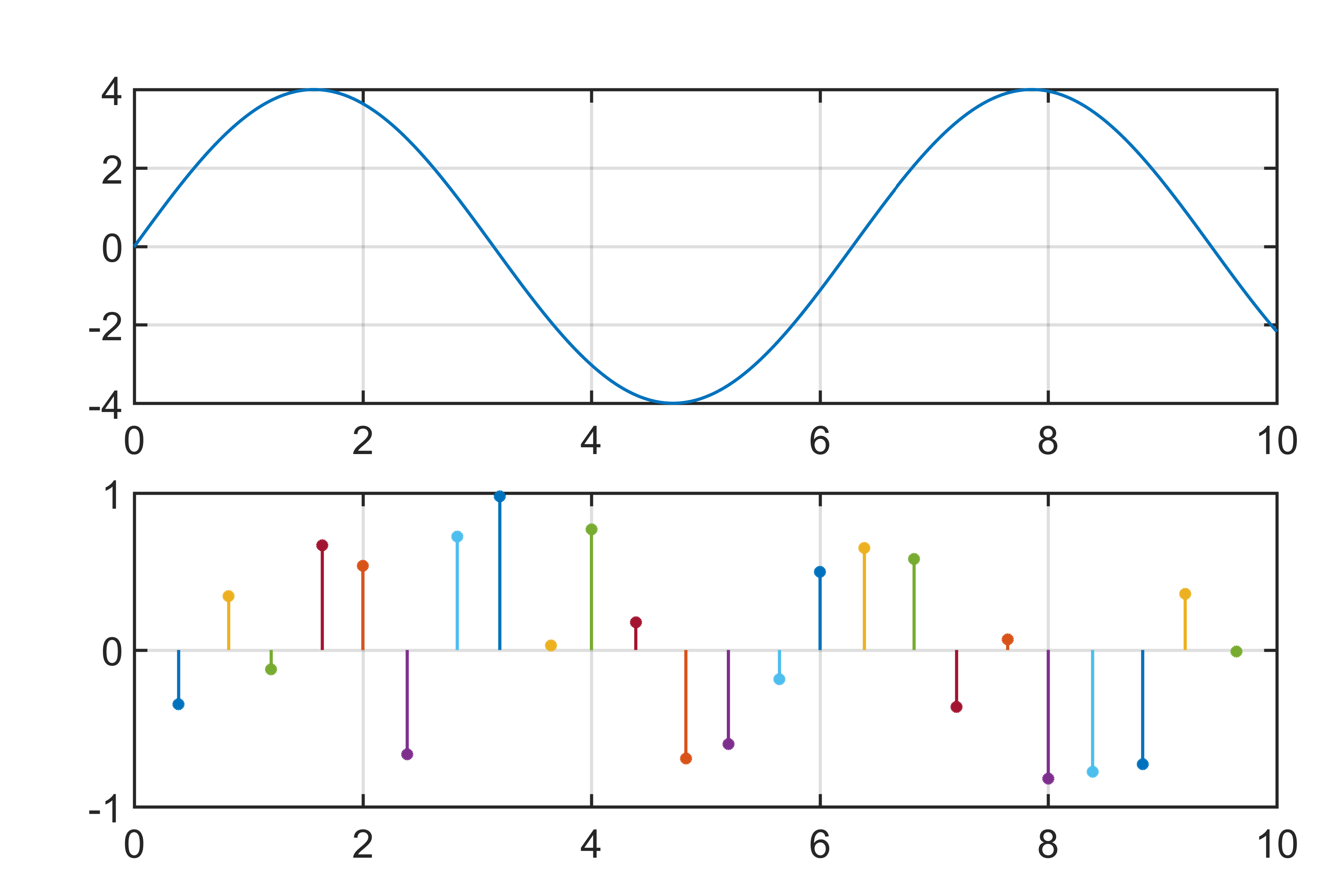}
  \caption{\textbf{Unconstrained scaling $\mu_c$.} Trajectory of the continuous-time input $w_c$ (top) and the discrete-time input $w_d$ (bottom)}\label{fig:inputs_rangeDT}
\end{figure}

\subsubsection{Minimum dwell-time}

Let us consider here the example from \cite{Briat:13d} to which we add disturbances as also done in \cite{Degue:16nolcos,Briat:18:IntImp}. We consider the system \eqref{eq:linimpdel} with the matrices:
\begin{equation}\label{eq:ex1}
\begin{array}{l}
    A=\begin{bmatrix}
    -1 & 0\\
    1 & -2
  \end{bmatrix}, G_c=\begin{bmatrix}
    0.5 & 0.1\\0 & 1
      \end{bmatrix}, E_c=\begin{bmatrix}
    0.1\\
    0.1
  \end{bmatrix},\\ J=\begin{bmatrix}
2 & 1\\
1 & 3
  \end{bmatrix}, G_d=\begin{bmatrix}
    0.1 & 0\\0 & 0.1
      \end{bmatrix}, E_d=\begin{bmatrix}
    0.3\\
    0.3
  \end{bmatrix},\\
    C_c=C_d=\begin{bmatrix}
    0 & 1
  \end{bmatrix}, H_c=H_d=\begin{bmatrix}
    0 & 0
  \end{bmatrix}, F_c=F_d=0.03.
\end{array}
\end{equation}
The disturbances and the delays are defined in the same way as in the previous example. Using a constant scaling $\mu_c$ in Theorem \ref{th:MDTinterval} with polynomials of degree 4, we get a minimum $\gamma$ of 0.19959. The computation time is 3.032 seconds and the number of primal and dual variables ia 424 and 120, respectively. The delays are $h_c=5$ and $h_d=4$, and the minimum dwell-time is set to $\bar T=1$. The simulation results are depicted in Fig.~\ref{fig:states_minDTZ} and Fig.~\ref{fig:inputs_minDTZ}.

\begin{figure}
  \centering
  \includegraphics[width=0.75\textwidth]{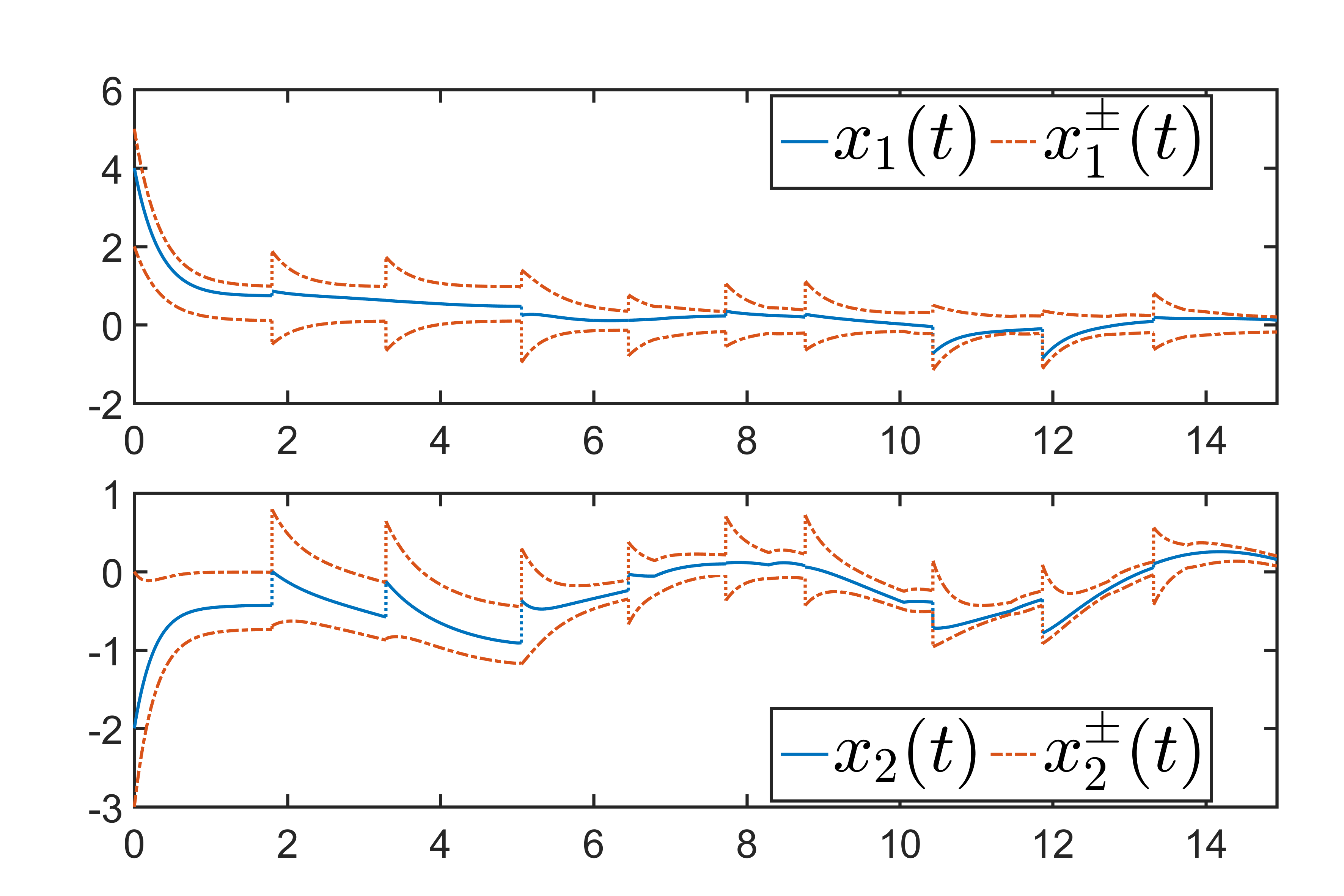}
  \caption{\textbf{Constant scaling $\mu_c$.} Trajectories of the system \eqref{eq:linimpdel}-\eqref{eq:ex1} and the interval observer \eqref{eq:obs}-\eqref{eq:L2}-\eqref{eq:Lex1} for some randomly chosen impulse times satisfying the minimum dwell-time $\bar{T}=1$.}\label{fig:states_minDTZ}
\end{figure}

\begin{figure}
  \centering
  \includegraphics[width=0.75\textwidth]{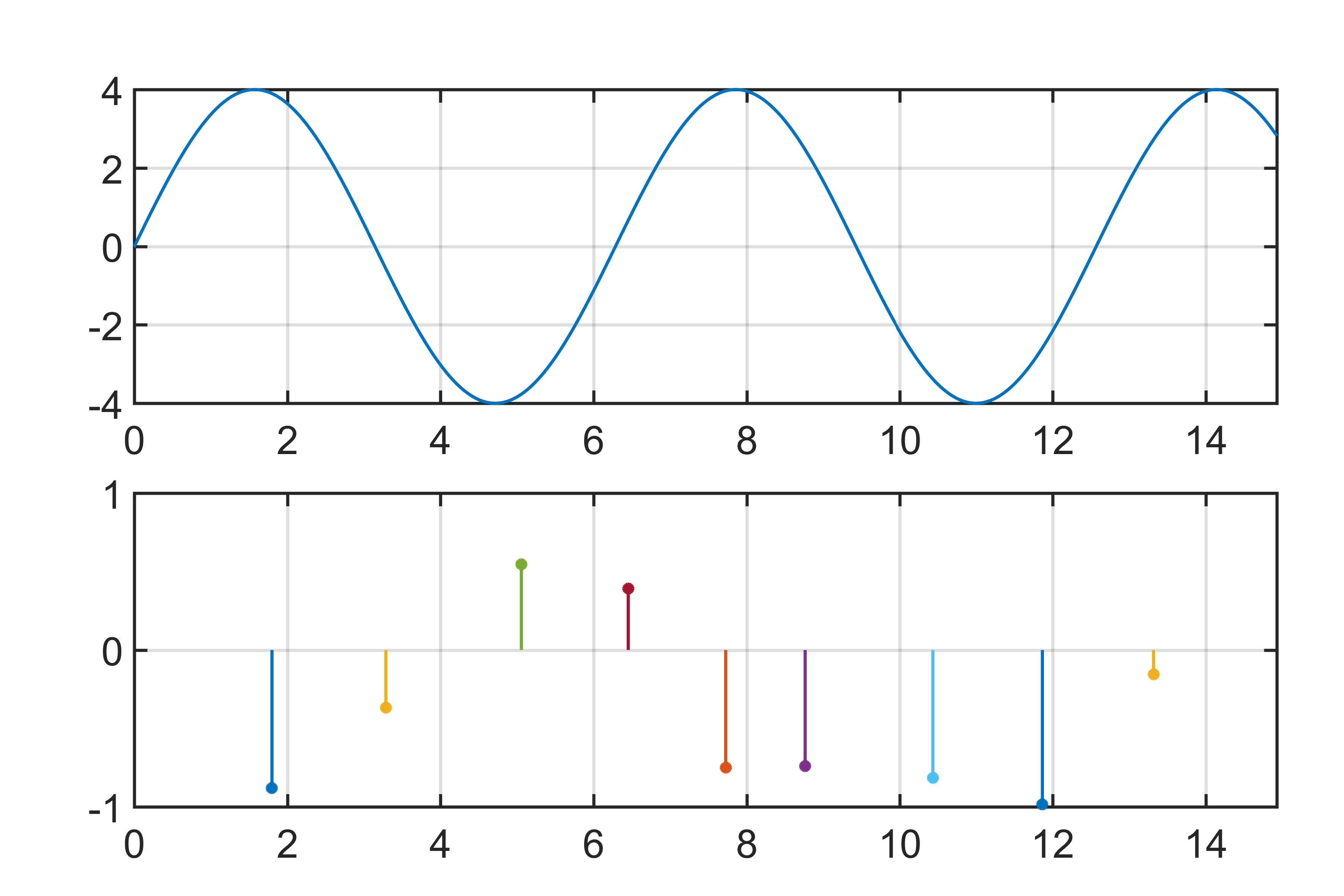}
  \caption{\textbf{Constant scaling $\mu_c$.} Trajectory of the continuous-time input $w_c$ (top) and the discrete-time input $w_d$ (bottom)}\label{fig:inputs_minDTZ}
\end{figure}

In the unconstrained scalings case, i.e. Corollary \ref{th:MDTinterval_unc}, in the same conditions as in the constant scaling case, we obtain 0.19957  for the minimum value for $\gamma$. It is interesting to note that this value is very close to the one obtained in the constant scaling case. The obtained gains are given  by
\begin{equation}\label{eq:Lex1}
  L_d=\begin{bmatrix}
   1\\
     1
  \end{bmatrix}\quad \textnormal{and}\quad \tilde L_c(\tau)=\begin{bmatrix}
    0\\
    3.3333
  \end{bmatrix}.
\end{equation}
For information, the semidefinite program has 432 primal variables, 120 dual variables and it takes 2.87 seconds to solve. Simulation results are depicted in Fig.~\ref{fig:states_minDT} and Fig.~\ref{fig:inputs_minDT}. Note the periodicity of the sequence of dwell-times.

\begin{figure}
  \centering
  \includegraphics[width=0.75\textwidth]{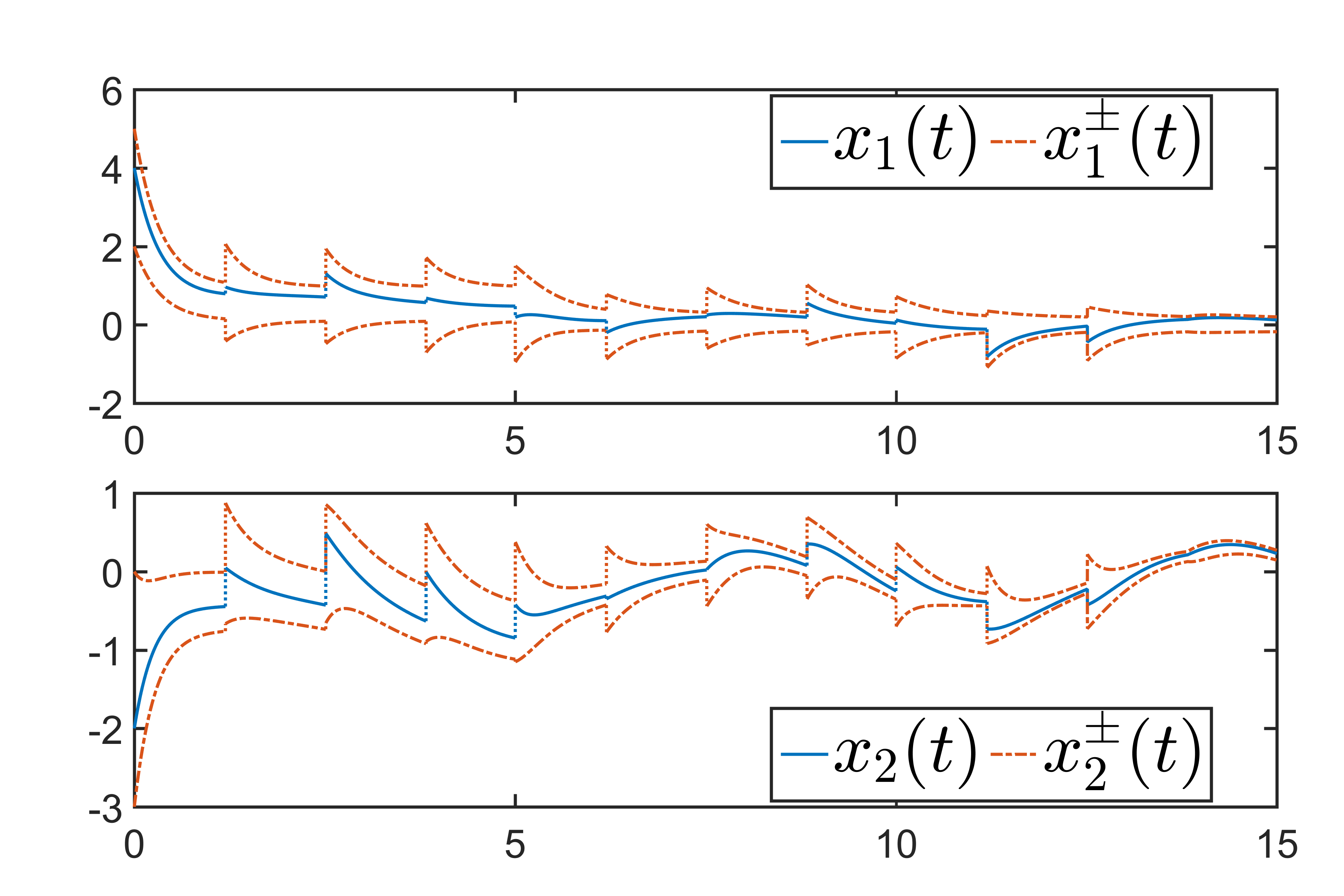}
  \caption{\textbf{Unconstrained scaling $\mu_c$.} Trajectories of the system \eqref{eq:linimpdel}-\eqref{eq:ex1} and the interval observer \eqref{eq:obs}-\eqref{eq:L2}-\eqref{eq:Lex1} for some randomly chosen impulse times satisfying the minimum dwell-time $\bar{T}=1$.}\label{fig:states_minDT}
\end{figure}

\begin{figure}
  \centering
  \includegraphics[width=0.75\textwidth]{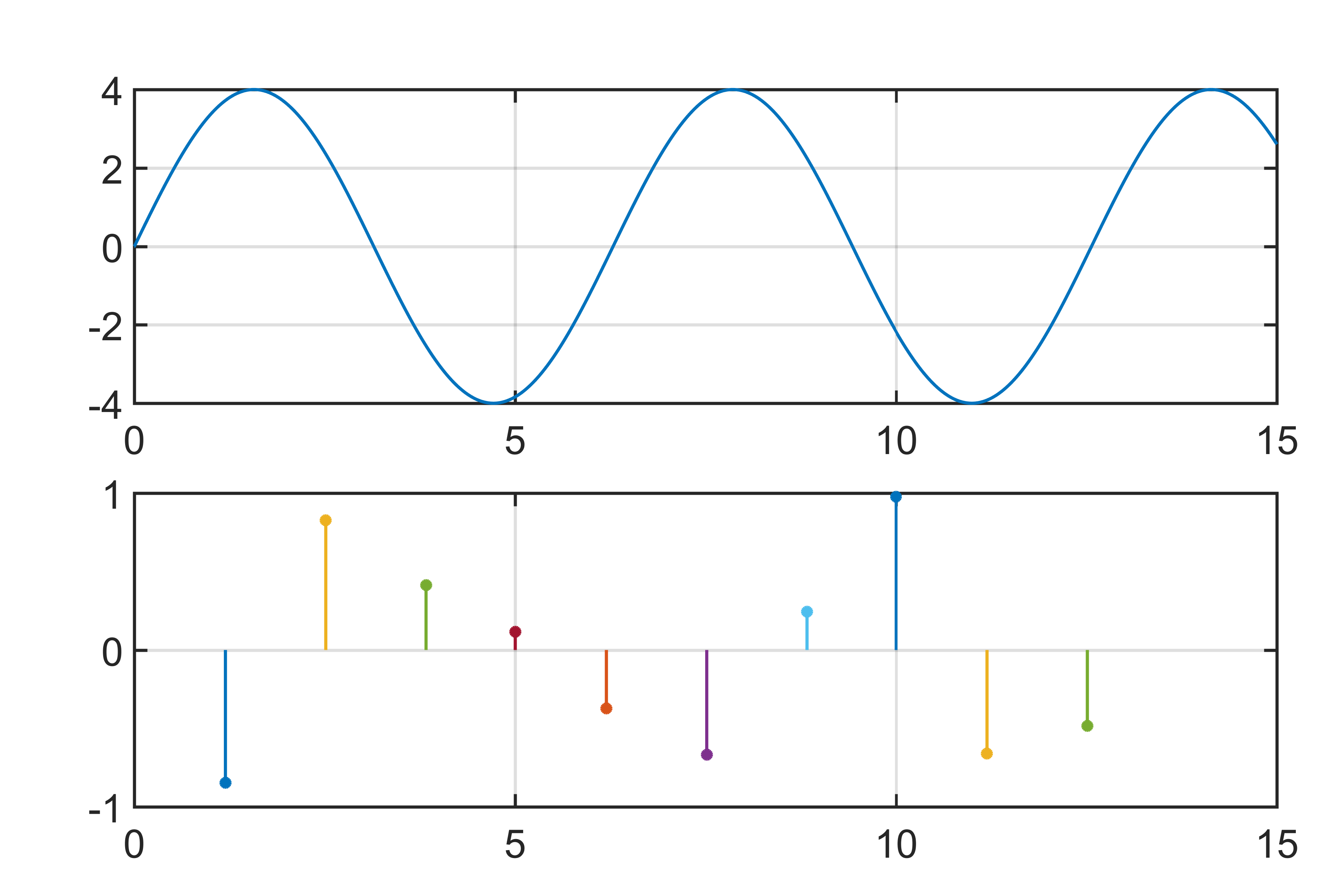}
  \caption{\textbf{Unconstrained scaling $\mu_c$.} Trajectory of the continuous-time input $w_c$ (top) and the discrete-time input $w_d$ (bottom)}\label{fig:inputs_minDT}
\end{figure}
%
%

\section{Interval observation of linear switched systems with delays}\label{sec:IIsw}

Let us consider here the switched system with delay
\begin{equation}\label{eq:switched}
\begin{array}{rcl}
    \dot{\tilde{x}}(t)&=&\tilde{A}_{\sigma(t)}\tilde{x}(t)+\tilde{G}_{\sigma(t)}x(t-h_c)+\tilde{E}_{\sigma(t)}w(t)\\
     \tilde{y}(t)&=&\tilde{C}_{\sigma(t)}\tilde{x}(t)+\tilde{H}_{\sigma(t)}x(t-h_c)+\tilde{F}_{\sigma(t)}w(t)
\end{array}
\end{equation}
where $\sigma:\mathbb{R}_{\ge0}\mapsto\{1,\ldots,N\}$ is the switching signal, $\tilde x\in\mathbb{R}^n$ is the state of the system, $\tilde{w}\in\mathbb{R}^p$ is the exogenous input and $\tilde{y}\in\mathbb{R}^p$ is the measured output. The switching signal $\sigma$ is assumed to take values in the set $\{1,\ldots,N\}$ and to change values at the times in the sequence $\{t_k\}_{k\ge1}$. This system can be rewritten into the following impulsive system with multiple jump maps as in  \cite{Briat:15i,Briat:16c}
\begin{equation}\label{eq:swimp}
\begin{array}{rcl}
    \dot{x}(t)&=&\diag_{i=1}^N(\tilde{A}_{i})x(t)+\diag_{i=1}^N(\tilde{G}_{i})x(t-h_c)+\col_{i=1}^N(\tilde{E}_{i})w(t)\\
     y(t)&=&\diag_{i=1}^N(\tilde{C}_{i})x(t)+\diag_{i=1}^N(\tilde{H}_{i})x(t-h_c)+\col_{i=1}^N(\tilde{F}_{i})w(t)\\
     x(t_k^+)&=&J_{ij}x(t_k),\ i,j=1,\ldots,N,\ i\ne j
\end{array}
\end{equation}
where $J_{ij}:=(b_ib_j^T)\otimes I_n$ and $\{b_1,\ldots,b_N\}$ is the standard basis for $\mathbb{R}^N$. It is important to stress that in the above formulation only the part of the state $x(t)$ that evolves according to the subsystem $\sigma(t)$ is meaningful. In this regard, the others can be discarded when plotting the trajectories of the switched system.

\subsection{Proposed interval observer}

Because of the particular structure of the system \eqref{eq:swimp}, we can define w.l.o.g.  an interval observer of the form
\begin{equation}\label{eq:obs2}
\begin{array}{lcl}
      \dot{x}^\bullet(t)&=&\diag_{i=1}^N(\tilde{A}_{i})x^\bullet(t)+\diag_{i=1}^N(\tilde{G}_{i})x^\bullet(t-h_c)+\col_{i=1}^N(\tilde{E}_{i})w_c^\bullet(t)\\
      &&+\diag_{i=1}^N(L_i(t))(y_c(t)-\tilde{C}_{\sigma(t)}x^\bullet(t)-\tilde{H}_{\sigma(t)}x^\bullet(t-h_c)-\col_{i=1}^N(\tilde{F}_{i})w_c^\bullet(t))\\
      x^\bullet(t_k^+)&=&J_{ij}x^\bullet(t_k),\ i,j=1,\ldots,N,\ i\ne j\\
      x^\bullet(s)&=&\phi_0^\bullet(s),s\in[-h_c,0]
\end{array}
\end{equation}
where the $L_c^i(t)$'s are the time-varying gains to design. The error dynamics is then given in this case by
\begin{equation}\label{eq:error_switched}
\begin{array}{rcl}
    \dot{e}^\bullet(t)&=&\diag_{i=1}^N(\tilde{A}_{i}-L_i(t)\tilde{C}_i)e^\bullet(t)+\diag_{i=1}^N(\tilde{G}_{i}-L_i(t)\tilde{H}_i)e^\bullet(t-h_c)\\
    &&+\col_{i=1}^N(\tilde{E}_{i}-L_i(t)\tilde{F}_i)\delta^\bullet(t)\\
   e^\bullet(t_k^+)&=&\left[(b_ib_j^T)\otimes I_n\right]e^\bullet(t_k)\\
   e_c^\bullet(t)&=&\left[I_n\otimes  M\right]e^\bullet(t)
    \end{array}
\end{equation}
where $M\in\mathbb{R}^{n\times n}_{\ge0}$ is a weighting matrix.

\subsection{Minimum dwell-time result}

As in the case of impulsive systems, we choose observer gains of the form
\begin{equation}\label{eq:L2}
  L_i(t)=\left\{\begin{array}{ll}
  \tilde{L}_i(t-t_k)& \textnormal{if }t\in(t_k,t_{k}+\tau]\\
  \tilde{L}_i(\bar T)& \textnormal{if }t\in(t_k+\bar{T},t_{k+1}]
  \end{array}\right.
\end{equation}
where the functions $\tilde{L}_i:\mathbb{R}_{\ge0}\mapsto\mathbb{R}^{n\times q_c}$ are to be determined. The observation problem is defined, in this case, as follows:
\begin{problem}\label{problem3}
Find an interval observer of the form \eqref{eq:swimp} (i.e. a matrix-valued function $L_c(\cdot)$ of the form \eqref{eq:L2}  such that the error dynamics \eqref{eq:error_switched} is
  \begin{enumerate}[(a)]
    \item state-positive, that is, for all $i=1,\ldots,N$ we have that
    \begin{itemize}
      \item $\tilde A_i-\tilde L_i(\tau) \tilde C_i$ is Metzler for all $\tau\in[0,\bar{T}]$,
      \item $\tilde E_i-\tilde L_i(\tau) \tilde F_i $ and  $\tilde G_i-\tilde L_i(\tau) \tilde H_i $ are nonnegative for all $\tau\in[0,\bar{T}]$,
    \end{itemize}
    \item asymptotically stable under minimum dwell-time $\bar T$ when $w_c\equiv0$, and
    \item the map $\delta^\bullet\mapsto e_c^\bullet$ has an $L_1$-gain of at most $\gamma$.
  \end{enumerate}
\end{problem}
The following result provides a sufficient condition for the solvability of Problem \ref{problem3}:
\begin{theorem}\label{th:3}\label{th:switched}
Assume that there exist differentiable matrix-valued functions $X_i:[0,\bar T]\mapsto\mathbb{D}^n$, $X_i(0)\succ0$, $X_i(\bar T)\succ0$, $i=1,\ldots,N$, a matrix-valued function $Y_i:[0,\bar T]\mapsto\mathbb{R}^{n\times q_c}$, $i=1,\ldots,N$,  $U\in\mathbb{D}^{n}_{\succ0}$ and scalars 
$\eps,\alpha,\gamma>0$ such that the conditions
\begin{subequations}\label{eq:th3a}
\begin{alignat}{4}
            X_i(\tau)\tilde A_i-Y_i(\tau)\tilde C_{i}+\alpha I_n&\ge0\label{eq:th3:1}\\
            X_i(\tau)\tilde G_i-Y_i(\tau)\tilde H_{i}&\ge0\label{eq:th3:2}\\
            X_i(\tau)\tilde E_i-Y_i(\tau)\tilde F_i&\ge0\label{eq:th3:3}
  \end{alignat}
\end{subequations}
           and
 \begin{subequations}\label{eq:th3b}
\begin{alignat}{4}
          \mathds{1}^T\left[X_i(\bar T)\tilde A_i-Y_i(\bar T)\tilde C_i+U\right]+\mathds{1}^TM&\le0\label{eq:th3:50}\\
           \mathds{1}^T\left[X_i(\bar T)\tilde G_i-Y_i(\bar T)\tilde H_{i}-U\right]&\le0\label{eq:th3:50}\\
            \mathds{1}^T\left[X_i(\bar T)\tilde E_i-Y_i(\bar T)\tilde F_i\right]-\gamma \mathds{1}^T&\le0\label{eq:th3:70}\\
           \mathds{1}^T\left[\dot{X}_i(\tau)+X_i(\tau)\tilde A_i-Y_i(\tau)\tilde C_{i}+U\right]+\mathds{1}^TM&\le0\label{eq:th3:5}\\
           \mathds{1}^T\left[X_i(\tau)\tilde G_i-Y_i(\tau)\tilde H_{i}-U\right]&\le0\label{eq:th3:5}\\
            \mathds{1}^T\left[X_i(\tau)\tilde E_i-Y_i(\tau)\tilde F_i\right]-\gamma \mathds{1}^T&\le0\label{eq:th3:7}\\
            \mathds{1}^T\left[X_i(0)-X_j(\bar T)+\eps I\right]&\le0\label{eq:th3:6}
             \end{alignat}
\end{subequations}
hold for all $\tau\in[0,\bar T]$ and for all $i,j=1,\ldots,N$, $i\ne j$. Then, there exists an interval observer of the form \eqref{eq:obs2} that solves Problem \ref{problem3} and suitable observer gains are given by
\begin{equation}\label{eq:formula2}
  \tilde{L}_i(\tau)= X_i(\tau)^{-1}Y_i(\tau)
\end{equation}
for all $i=1,\ldots,N$.
\end{theorem}
\blue{\begin{proof}
  The proof is an application of Theorem \ref{th:minimumDT:general} to the error system \eqref{eq:error_switched}.
\end{proof}}
\begin{remark}
Note that since the continuous-time scaling needs to be independent of the timer variable, then it cannot depend on the mode of the switched system.
\end{remark}
\blue{The following result is the unconstrained scalings counterpart of Theorem \ref{th:switched}:}
\begin{corollary}[Minimum dwell-time - Unconstrained scalings case]\label{cor:switched}
Assume that there exist differentiable matrix-valued functions $X_i:[0,\bar T]\mapsto\mathbb{D}^n$, $X_i(0)\succ0$, $X_i(\bar T)\succ0$, $i=1,\ldots,N$, a matrix-valued function $Y_i:[0,\bar T]\mapsto\mathbb{R}^{n\times q_c}$, $i=1,\ldots,N$,  $U_i\in\mathbb{D}^{n}_{\succ0}$, $i=1,\ldots,N$, and scalars $\eps,\alpha,\gamma>0$ such that the conditions
\begin{subequations}
\begin{alignat}{4}
            X_i(\tau)\tilde A_i-Y_i(\tau)\tilde C_{i}+\alpha I_n&\ge0\label{eq:th3:1}\\
            X_i(\tau)\tilde G_i-Y_i(\tau)\tilde H_{i}&\ge0\label{eq:th3:2}\\
            X_i(\tau)\tilde E_i-Y_i(\tau)\tilde F_i&\ge0\label{eq:th3:3}
  \end{alignat}
\end{subequations}
           and
 \begin{subequations}
\begin{alignat}{4}
          \mathds{1}^T\left[X_i(\bar T)\tilde A_i-Y_i(\bar T)\tilde C_i+X_i(\bar T)\tilde G_i-Y_i(\bar T)\tilde H_{i}\right]+\mathds{1}^TM&\le0\\
            \mathds{1}^T\left[X_i(\bar T)\tilde E_i-Y_i(\bar T)\tilde F_i\right]-\gamma \mathds{1}^T&\le0\\
           \mathds{1}^T\left[\dot{X}_i(\tau)+X_i(\tau)\tilde A_i-Y_i(\tau)\tilde C_{i}+X_i(\tau)\tilde G_i-Y_i(\tau)\tilde H_{i}\right]+\mathds{1}^TM&\le0\label{eq:th3:5}\\
            \mathds{1}^T\left[X_i(\tau)\tilde E_i-Y_i(\tau)\tilde F_i\right]-\gamma \mathds{1}^T&\le0\\
            \mathds{1}^T\left[X_i(0)-X_j(\bar T)+\eps I\right]&\le0
             \end{alignat}
\end{subequations}
hold for all $\tau\in[0,\bar T]$ and for all $i,j=1,\ldots,N$, $i\ne j$. Then, there exists an interval observer of the form \eqref{eq:obs2} that solves Problem \ref{problem3} with the additional restriction that sequence of dwell-times belongs to
 \begin{equation}
    \mathcal{T}_{MDT,p}^\sigma:=\left\{\{T_0,T_1,\ldots\}\left|\begin{array}{c}
    T_{qi+k}=\beta_k\ge\bar T,\sigma(t_{qi+k}):=\delta_k\in\{1,\ldots,N\},\\
    k\in\{0,\ldots,q-1\},    \sum_{i=0}^{q-1}T_i=h_c/\alpha,\\
    (q,\alpha,i)\in\mathbb{Z}_{>0}^2\times\mathbb{Z}_{\ge0},h_c>0
    \end{array}\right.\right\}.
  \end{equation}
Moreover, suitable observer gains are given by
\begin{equation}
  \tilde{L}_i(\tau)= X_i(\tau)^{-1}Y_i(\tau)
\end{equation}
for all $i=1,\ldots,N$.
\end{corollary}

\begin{remark}
It is interesting to note that in the case of switched systems, then both the sequence of dwell-times and the sequence of switching signal values need to satisfy the periodicity property. This adds some restrictions on the possibility of using timer-dependent continuous-time scalings.
\end{remark}

\subsection{Examples}

\subsubsection{Example 1. Toy model}

Let us consider the system \eqref{eq:switched} with the matrices.
\begin{equation}\label{eq:ex3}
\begin{array}{l}
    \tilde A_1=\begin{bmatrix}
    -1 & 0\\
    1 & -2
  \end{bmatrix}, \tilde G_1=\begin{bmatrix}
    0.1 & 0\\
    1 & 0.5
  \end{bmatrix}, \tilde E_1=\begin{bmatrix}
    0.1\\
    0.1
  \end{bmatrix},\\ \tilde A_2=\begin{bmatrix}
-1 & 1\\
1 & -6
  \end{bmatrix}, \tilde G_2=\begin{bmatrix}
    0 & 0\\
    0 & 2
  \end{bmatrix}, \tilde E_2=\begin{bmatrix}
    0.5\\
    0
  \end{bmatrix},
    \\ \tilde C_1=\tilde C_2=\begin{bmatrix}
    0 & 1
  \end{bmatrix},  \tilde H_1=\tilde H_2=\begin{bmatrix}
    0 & 0
  \end{bmatrix}, \tilde F_c=\tilde F_d=0.1.
\end{array}
\end{equation}
Solving for the conditions in Theorem \ref{th:switched} with a constant scaling $\mu_c$, polynomials of degree 4 and a minimum dwell-time equal to $\bar T=1$, we get the minimum value 1.3338 for $\gamma$. The problem solves in  7.44 seconds and the number of primal/dual variables is 789/210. The following gains are obtained.
\begin{equation}
  \tilde L_1(\tau) = \begin{bmatrix}
    0\\1
  \end{bmatrix}\ \textnormal{ and } \tilde L_2(\tau) = \begin{bmatrix}
    \dfrac{-2.1270\tau^4   -0.0797\tau^3   -1.3068\tau^2   -1.3975\tau  -27.1294}{1.7779\tau^4   -0.9203\tau^3   -0.0802\tau^2   -1.0925\tau  -31.7254}\\
    \dfrac{2.4707\tau^4+    0.7914\tau^3+    0.9126\tau^2   -1.2046\tau   -2.9700}{ 0.1266\tau^4  +  0.7598\tau^3    +0.0874 \tau^2+   0.2920\tau+    3.7814}
  \end{bmatrix}
\end{equation}
For simulation purposes, we select $h_c=5$ and $h_d=4$ and we get the results depicted in Fig.~\ref{fig:states_minDT_switchedZ} and Fig.~\ref{fig:inputs_minDT_switchedZ}.

\begin{figure}
  \centering
  \includegraphics[width=0.75\textwidth]{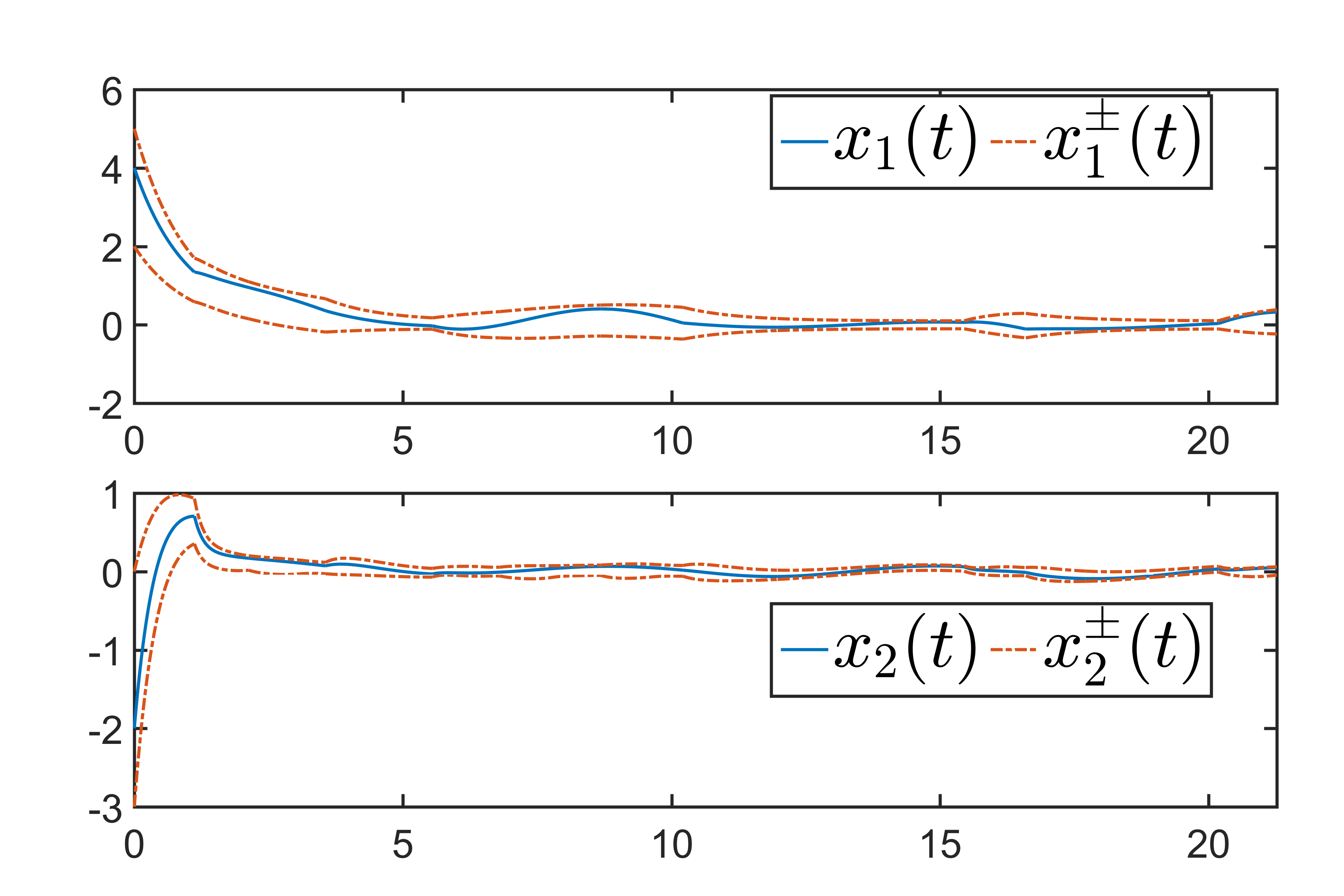}
  \caption{\textbf{Constant scaling $\mu_c$.} Trajectories of the system \eqref{eq:switched}-\eqref{eq:ex3} and the interval observer \eqref{eq:swimp} for some randomly chosen impulse times satisfying the minimum dwell-time $\bar T=1$.}\label{fig:states_minDT_switchedZ}
\end{figure}

\begin{figure}
  \centering
  \includegraphics[width=0.75\textwidth]{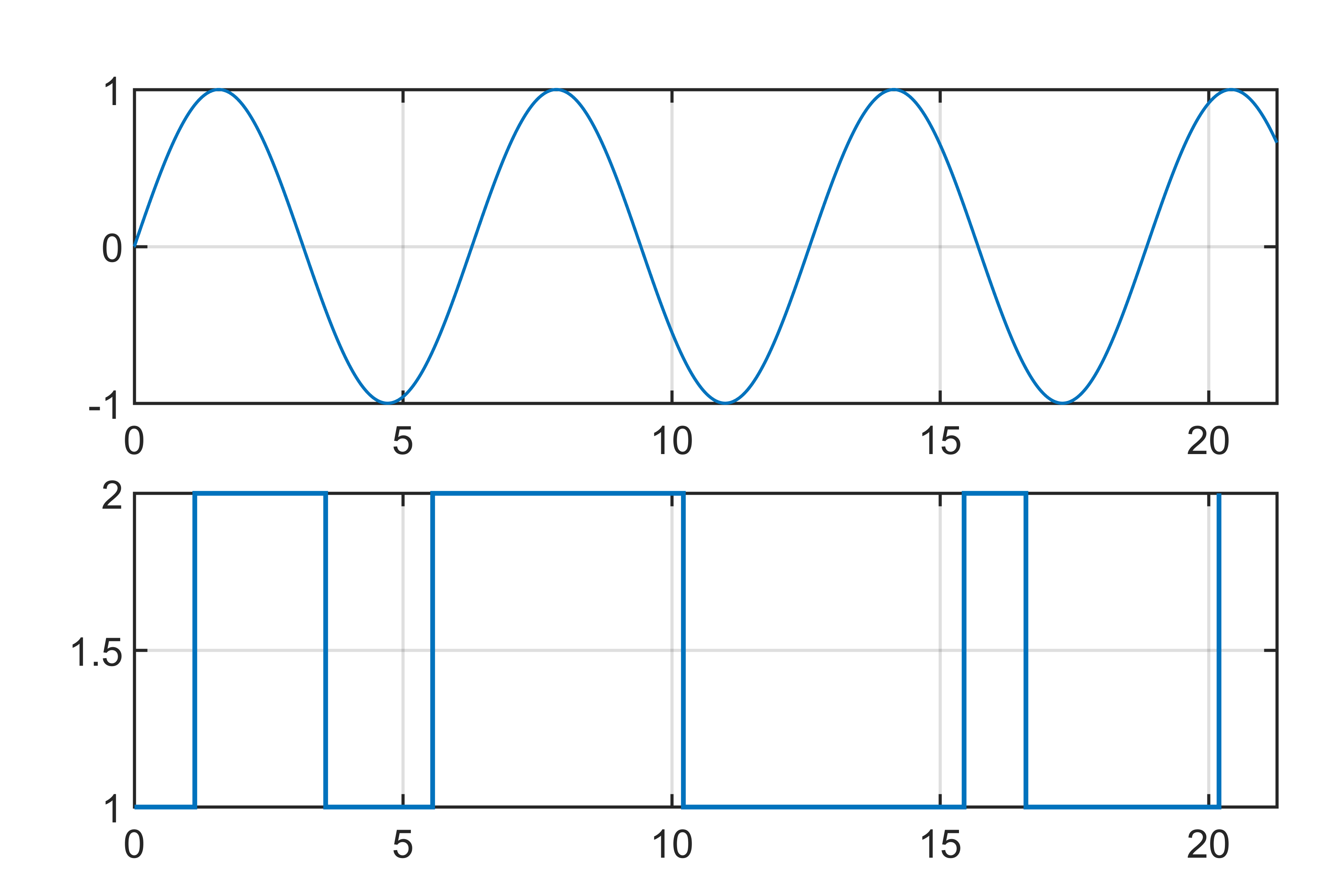}
  \caption{\textbf{Constant scaling $\mu_c$.} Trajectory of the continuous-time input $w_c$ (top) and the switching signal $\sigma$ (bottom)}\label{fig:inputs_minDT_switchedZ}
\end{figure}

We now consider the Corollary \ref{cor:switched} and we get the minimum $\gamma=0.8002$. The number of primal/dual variables is 807/210 and the problem is solved in 5.727 seconds. The observer gains are given by
\begin{equation}
  \tilde L_1(\tau) = \begin{bmatrix}
    0\\1
  \end{bmatrix}\ \textnormal{ and } \tilde L_2(\tau) = \begin{bmatrix}
    \dfrac{-2.1617\tau^4+    0.9726\tau^3+    0.4774\tau^2   -3.1312\tau  -35.0211}{ 2.5885\tau^4   -0.8470\tau^3+    0.5226\tau^2   -4.1326\tau  -36.9951}\\
    \dfrac{-15.8770\tau^4  -1.6704\tau^3    0.7153\tau^2   11.1118\tau+    5.7205}{2.9314\tau^4   -0.5488\tau^3   -0.2847\tau^2   -4.2896  \tau -15.2677}\\
  \end{bmatrix}
\end{equation}
The trajectories of the system and the interval observer are depicted in Fig.~\ref{fig:states_minDT_switched}. The disturbance input and the switching signal are depicted in Fig.~\ref{fig:inputs_minDT_switched}.

\begin{figure}
  \centering
  \includegraphics[width=0.75\textwidth]{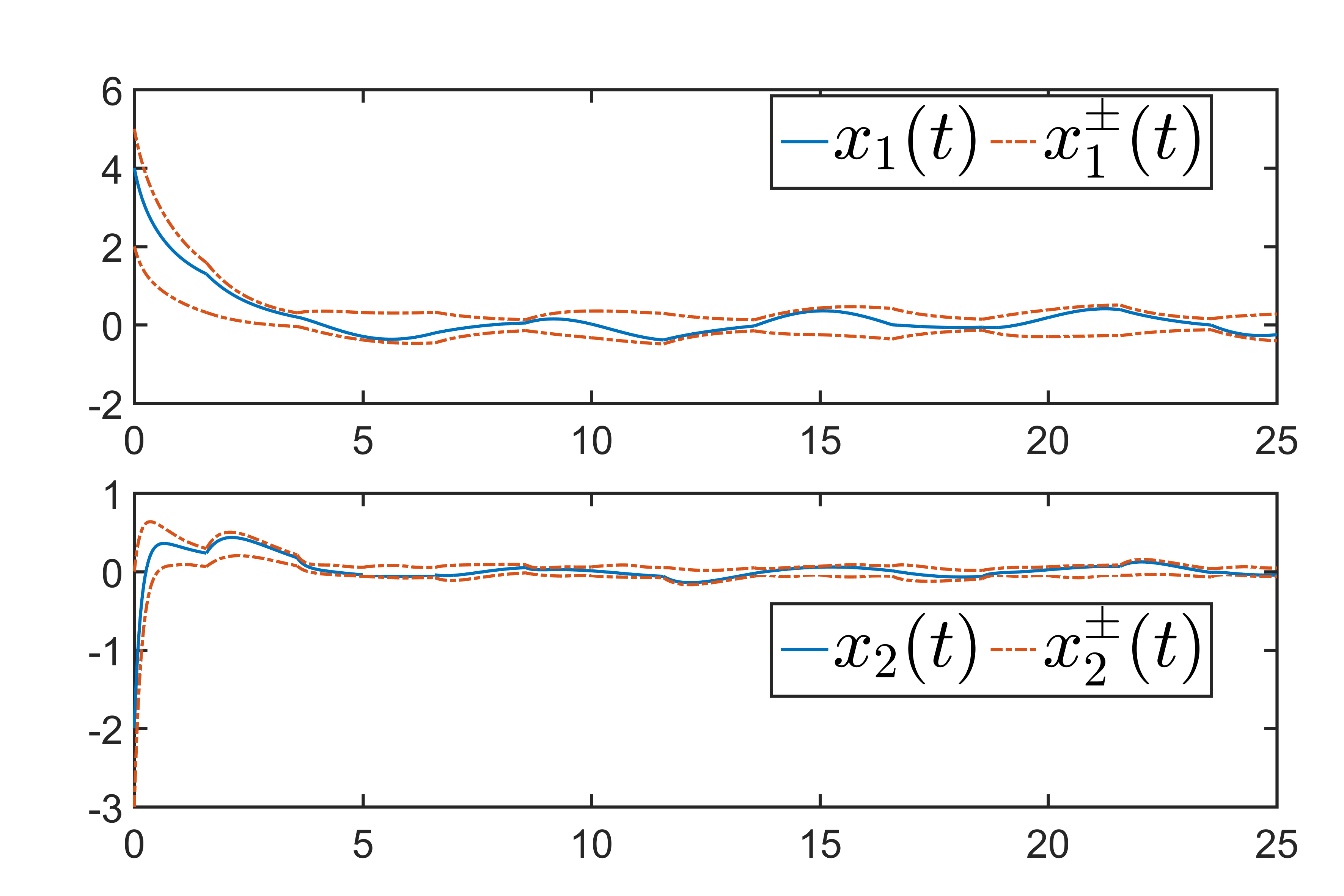}
  \caption{\textbf{Unconstrained scaling $\mu_c$.} Trajectories of the system \eqref{eq:switched}-\eqref{eq:ex3} and the interval observer \eqref{eq:swimp} for some randomly chosen impulse times satisfying the minimum dwell-time $\bar T=1$.}\label{fig:states_minDT_switched}
\end{figure}

\begin{figure}
  \centering
  \includegraphics[width=0.75\textwidth]{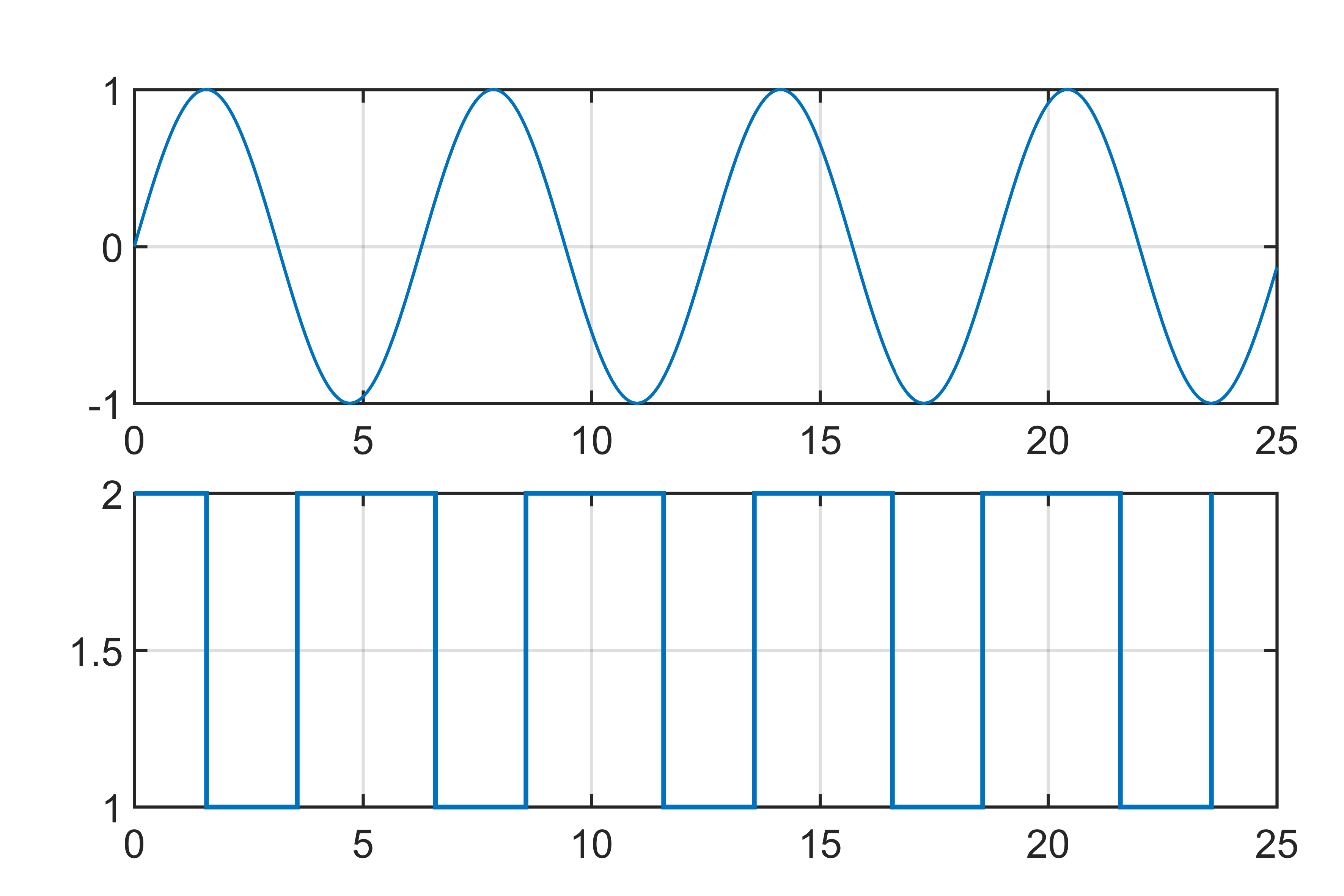}
  \caption{\textbf{Unconstrained scaling $\mu_c$.} Trajectory of the continuous-time input $w_c$ (top) and the switching signal $\sigma$ (bottom)}\label{fig:inputs_minDT_switched}
\end{figure}

\blue{ \subsubsection{Example 2. Foschini-Miljanic algorithm}

 The Foschini-Miljanic algorithm \cite{Foschini:93} is a well-known algorithm which  provides distributed on-line power control of wireless networks with user-specific Signal-to-Interference-and-Noise-Ratio (SINR) requirements. This algorithm notably yields the minimum transmitter powers that satisfy these requirements. It is described by the following dynamical system
 \begin{equation}
   \dot{p}_i(t)=\kappa_i\left[-p_i(t)+\gamma_i\left(\sum_{j=1,j\ne i}^n\dfrac{g_{ij}}{g_{ii}}p_j(t)+\dfrac{\nu_i}{g_{ii}}\right)\right]
 \end{equation}
 where $\kappa_i > 0$ denote the proportionality constants and $\gamma_i$ denote the desired SINR. The constants $g_{ij}$ and $\nu_i$ are related to interference and the thermal noise; see \cite{Foschini:93,Koskie:06} for more details. It has been recently refined in order to incorporate delays and switching topologies
 \begin{equation}
   \dfrac{\d p_i(t)}{\d t}=\kappa_i\left[-p_i(t)+\gamma_i\left(\sum_{j=1,j\ne i}^n\dfrac{g_{ij}^{\sigma(t)}}{g_{ii}^{\sigma(t)}}p_j(t-\tau_j(t))+\dfrac{\nu_i^{\sigma(t)}}{g_{ii}^{\sigma(t)}}\right)\right]
 \end{equation}
 where the $\tau_i$'s are the time-varying delays and $\sigma$ is a switching signal that changes the communication topology; see e.g. \cite{Charalambous:08,Zappavigna:12}. The above system can be compactly rewritten as
\begin{equation}
  \dot{p}(t)=-Kp(t)+K\left(\sum_{k=1}^nB_k^{\sigma(t)}p\left(t-\tau_k(t)\right)+\eta^{\sigma(t)}\right).
\end{equation}
For simplicity, let us consider $K=I$, one single constant delay $\tau$ and 3 nodes (i.e. $n=3$). This yields
\begin{equation}
  \dot{p}(t)=-p(t)+B^{\sigma(t)}p(t-\tau)+\eta^{\sigma(t)}.
\end{equation}

 For numerical purposes, we consider the following matrices $A_1=A_2=-I$,
 \begin{equation}
   G_1=B^1=\begin{bmatrix}
     0 & 0.675 & 0.3\\
     0.375 & 0 & 0.15\\
     0.45 & 0.75 & 0
   \end{bmatrix},G_2=B^2=\begin{bmatrix}
     0 & 0.5 & 0.6\\
     0.9 & 0 & 0.1\\
     0.2 & 1.2 & 0
   \end{bmatrix}
 \end{equation}
 together with
 \begin{equation}
   E_1=E_2=I,C_1=C_2=\begin{bmatrix}
     1 & 0 & 0
   \end{bmatrix},H_1=H_2=0,F_1=F_2=0,M=I.
 \end{equation}
 In other words, we would like to estimate upper and lower bonds on the state of the system by just measuring the state of the first note. Solving for the conditions of Theorem \ref{th:switched} with a constant scaling $\mu_c$, polynomials of degree 2 and a minimum dwell-time equal to $\bar T=0.2$, we get the minimum value 3.074 for $\gamma$. The problem solves in  13.2 seconds and the number of primal/dual variables is 956/394. The following gains are obtained\footnote{We add here extra conditions to constrain the values of the entries of the gains to lie within the interval $[-10,10]$. See \cite{Briat:11h} for more details.}
 \begin{equation}
   \tilde L_1(\tau)=\begin{bmatrix}
    10\\0\\0
   \end{bmatrix}\ \textnormal{and}\ \tilde L_2(\tau)=\begin{bmatrix}
    10\\0\\0
   \end{bmatrix}.
 \end{equation}}

\section{Conclusion and future works}

Several stability and performance analysis conditions for the stability analysis of a class of uncertain linear positive systems with impulses have been obtained for the first time using an input/output approach. Interestingly, the scalings can be made timer-dependent but the fact that impulses arrive aperiodically makes their use difficult. In most of the interesting cases, the continuous-time scalings need to be timer-independent. It is shown that in the case of timer-dependent scalings, the obtained conditions are exactly the stability conditions for the system with delays equal to 0, which is reminiscent of many existing results in the literature. The conditions for the design of interval observers can be easily extended to cope with multiple constant delays and other types of delays such as time-varying delays, distributed delays, neutral systems, etc. along the lines of the work in \cite{Briat:16b}. The consideration of other performance measures such as the $L_\infty$-gain is also of crucial interest here as this would allow for the derivation of results for systems with time-varying delays which would not depend on the delay-derivative; see e.g. \cite{Briat:11j,Briat:11h,Zhu:15,Briat:book1,Briat:16b}.


\end{document}